\documentclass[leqno,12pt]{amsart} 
\usepackage{amssymb}
\usepackage[dvips]{graphicx}
\usepackage{here}
\usepackage{pifont}
\usepackage{booktabs}
\usepackage{mathtools}
\mathtoolsset{showonlyrefs=true}
\numberwithin{equation}{section}
\theoremstyle{plain} 
\newtheorem{theorem}{\indent\sc Theorem}[section] 
\newtheorem{lemma}[theorem]{\indent\sc Lemma}
\newtheorem{corollary}[theorem]{\indent\sc Corollary}

\theoremstyle{definition} 

\newtheorem{remark}[theorem]{\indent\sc Remark}

\newtheorem{question}{\indent\sc Question}
\newcommand{\R}{{\mathbb R}}
\newcommand{\C}{{\mathbb C}}

\newcommand{\N}{{\mathbb N}}
\newcommand{\Z}{{\mathbb Z}}
\renewcommand{\a}{\alpha}
\renewcommand{\b}{\beta}
\renewcommand{\d}{\partial}

\begin{document}

\title[Local zeta functions]{Meromorphic continuation and
non-polar singularities of local zeta functions in some smooth cases} 

\author[T. Nose]{Toshihiro Nose} 




\renewcommand{\thefootnote}{\fnsymbol{footnote}}
\footnote[0]{2020\textit{ Mathematics Subject Classification}.
Primary 30E15; Secondary 14B05, 14M25}

\thanks{ 
$^*$This work was partially supported by 
JSPS KAKENHI Grant Numbers JP19K14563.
}
\bigskip
\address{ 
Faculty of Engineering, Fukuoka Institute of Technology, 
Wajiro-higashi 3-30-1, Higashi-ku, Fukuoka, 811-0295, Japan
}
\email{nose@fit.ac.jp}

\maketitle

\begin{abstract}
It is known that local zeta functions associated with real analytic functions can be analytically continued as meromorphic functions to the whole complex plane.
In this paper, certain cases of specific (non-real analytic) smooth functions are precisely investigated.
In particular, we give asymptotic limits of local zeta functions 
at some singularities along one direction.
It follows from these behaviors that the respective local zeta functions have singularities different from poles.
Then we show the 
optimality of the lower estimates of a certain quantity 
concerning with meromorphic continuation of local zeta functions 
in some smooth model cases. 
\end{abstract}

\section{Introduction}

We consider integrals of the form
\begin{equation}\label{eqn:1.1}
Z_f(\varphi)(s)=\int_{\R^2}|f(x,y)|^s\varphi(x,y)dxdy
\quad \mbox{for $s\in\C$},
\end{equation}
where $f,\varphi$ are real-valued ($C^{\infty}$) smooth functions defined on a small open neighborhood $U$ of the origin in $\R^2$,
and the support of $\varphi$ is contained in $U$.
Since the integrals $Z_f(\varphi)(s)$ converge locally uniformly on the half-plane ${\rm Re}(s)>0$,
they become holomorphic functions there,
which are sometimes called {\it local zeta functions}.
The purpose of this paper is to study the region to which $Z_f(\varphi)(s)$ can be analytically continued.
In this paper, we consider the case where
$f(0,0)=0$ and $\nabla f(0,0)=0$.
Unless the above conditions hold, every problem treated in this paper is easy.

When $f$ is real analytic,
by using Hironaka's resolution of singularities \cite{hir64},
it is shown in \cite{bg69}, \cite{ati70} that
$Z_f(\varphi)(s)$ can be analytically continued as meromorphic functions
to the whole complex plane and their poles are contained in finitely many arithmetic progressions which consist of negative rational numbers.
More precisely, Varchenko \cite{var76} determines the exact location of poles of local zeta functions by using the theory of toric varieties based on the {\it Newton polyhedron} of $f$ under some nondegeneracy condition.
We remark that the above-mentioned results can be naturally extended in the general dimensional case.
Until now,
there have been many precise investigations
\cite{agv88}, \cite{ds89}, \cite{ds92}, \cite{dns05}, \cite{gre10jam}, \cite{ot13}, \cite{cgp13} 
in the real analytic case.
Kamimoto and the author \cite{kn16tor}, \cite{kn16non} generalize Varchenoko's results when 
$f$ belongs to a certain class of smooth functions.

When $f$ is a smooth function,
it may occur that
$Z_f(\varphi)(s)$ cannot be meromorphically continued to the whole complex plane.
In fact, it is shown in \cite{kn19} that
$Z_f(\varphi)(s)$ has a singularity different from poles
in the case of specific smooth functions.
(See \eqref{eqn:2.2} below.)
To investigate such failure of meromorphy of local zeta functions,
Kamimoto and the author
introduce the following quantities in \cite{kn20}:
\begin{equation}
\begin{split}
m_0(f,\varphi)&:=\sup\left\{\rho>0: 
\begin{array}{l} 
\mbox{The domain in which $Z_f(\varphi)$ can} \\ 
\mbox{be meromorphically continued} \\
\mbox{contains the half-plane ${\rm Re}(s)>-\rho$}
\end{array}
\right\},\\
m_0(f)&:=\inf\left\{m_0(f,\varphi):
\varphi\in C_0^{\infty}(U)\right\}.
\end{split}
\end{equation}
Note that $m_0(f,\varphi)=m_0(f)=\infty$ if $f$ is real analytic.
On the other hand,
if $Z_f(\varphi)(s)$ has a singularity different from poles,
then $m_0(f,\varphi)$ and $m_0(f)$ are finite.
In the same paper
\cite{kn20},
the case where $f$ can be expressed in the following form is investigated:
\begin{equation}\label{eqn:1.5}
f(x,y)=u(x,y)x^ay^b+(\mbox{flat function at the origin}),
\end{equation}
where $a,b\in \Z_+$ and $u$ is a smooth function defined near the origin with $u(0,0)\neq 0$.
Here, we say a smooth function defined near the origin is {\it flat} at the origin if all the derivatives of the function vanish at the origin.
Through a classification of \eqref{eqn:1.5} 
with respect to the flat function,
a nontrivial lower estimate of $m_0(f)$ is obtained in a certain case in \cite{kn20}.
(See Theorem \ref{thm:2.2} below.)
We emphasize that
the above case \eqref{eqn:1.5} might be considered as a natural model in the smooth case.
Indeed, the monomial case is essentially important in the real analytic case.
(See \cite{var76}, \cite{agv88}.)
Furthermore, Kamimoto \cite{kam22}
recently 
obtains general lower estimates of $m_0(f)$ for non-flat smooth functions $f$
by constructing {\it almost resolution of singularities} of zero varieties of smooth functions in two dimensions and showing that
any non-flat smooth functions can be expressed as \eqref{eqn:1.5} locally.

The purpose of this paper is to show the optimality of known lower estimates concerning to $m_0(f)$,
where $f$ is as in \eqref{eqn:1.5}. (See Table 1 and Question 1 below.)
To show the optimality,
we investigate the asymptotic behavior of local zeta functions associated with certain smooth functions which are not real analytic.
It is interesting in our results that local zeta functions simultaneously have poles
and non-polar singularities in a certain case.
On the other hand, few properties are known about such non-polar singularities, and it is not even known whether they are isolated singularities or not. (See Question 4 below.)

This paper is organized as follows.
In Section 2, we recall known results for analytic continuation of local zeta functions concerning to certain quantities.
We see a classification of smooth model cases
and lower estimates of the quantities for each case of the classification.
In Section 3, we state our main results.
In order to prove the main results,
we give some auxiliary lemmas in Section 4.
Sections 5--7 are devoted to the proof of 
Theorem \ref{thm:3.1} below.
In particular, analytic continuation of associated integrals are investigated in Section 5
and asymptotic limits of slightly different integrals
are given in Section 6.
In Section 8, we give a proof of the remainder of our results.

{\it Notation and symbols}
\begin{itemize}
\item
We denote by $\Z_{+}$ and $\R_{+}$ the subsets of all nonnegative numbers in $\Z$ and $\R$ respectively.
\item
We denote $C_0^{\infty}(U)$ as the set of all 
smooth functions
on $U$ whose support is compact and contained in $U$.
\item
We use the symbol
\begin{equation}\label{eqn:1.6}
\sigma:={\rm Re}(s)
\quad \mbox{for $s\in \C$} 
\end{equation}
which is traditionally used in the analysis of the Riemann zeta function.
\item
We denote by $\Gamma$ the gamma function.
\end{itemize}

\section{Known results}

In this section, we discuss certain quantities related to analytic continuation of local zeta functions.
These quantities (and local zeta functions) can be naturally generalized to the general dimensional case.
In Sections \ref{sec:2.1} and \ref{sec:2.2} below, these quantities have analogous properties in the general dimensional case if no specific mention is made.

\subsection{Holomorphic continuation of local zeta functions}\label{sec:2.1}
Let $f$, $\varphi$ be as in \eqref{eqn:1.1}.
We consider the following quantities.
\begin{equation}\label{eqn:2.0}
\begin{split}
h_0(f,\varphi)&:=\sup\left\{\rho>0: 
\begin{array}{l} 
\mbox{The domain in which $Z_f(\varphi)$ can} \\ 
\mbox{be holomorphically continued} \\
\mbox{contains the half-plane ${\rm Re}(s)>-\rho$}
\end{array}
\right\},\\
h_0(f)&:=\inf\left\{h_0(f,\varphi):
\varphi\in C_0^{\infty}(U)\right\}.
\end{split}
\end{equation}
It is known that the situation of analytic continuation of local zeta functions are observed through the integrability of the integral $Z_f(\varphi)(s)$,
which is the reason for considering half-planes as above.
(Needless to say, various kinds of region should be dealt with in the future.)
Indeed, holomorphic continuation of $Z_f(\varphi)(s)$ relates to the famous index
\begin{equation}
c_0(f):=\sup
\left\{\mu>0:
\begin{array}{l} 
\mbox{there exists an open neighborhood $V$ of} \\
\mbox{the origin in $U$ such that $|f|^{-\mu}\in L^1(V)$}
\end{array}
\right\},
\end{equation}
which is called 
{\it log canonical threshold} or {\it critical integrability index}.
Since $Z_f(\varphi)(s)$ converges locally uniformly on the half-plane ${\rm Re}(s)>-c_0(f)$,
$Z_f(\varphi)(s)$ becomes a holomorphic function there.
From the above, we have
\begin{equation}
c_0(f)\leq h_0(f) \leq h_0(f,\varphi).
\end{equation}
Moreover, we see from Theorem 5.1 in \cite{kn19} that if $\varphi$ satisfies
\begin{equation}\label{eqn:2.1}
\varphi(0,0)>0\quad \mbox{and}\quad
\varphi(x,y)\geq 0 \mbox{ on } U,
\end{equation}
then $Z_f(\varphi)(s)$ cannot be holomorphically continued
to any open neighborhood of $s=-c_0(f)$.
In particular, we have $h_0(f,\varphi)= c_0(f)$
under the condition \eqref{eqn:2.1},
which implies that the equality:
\begin{equation}
h_0(f)=c_0(f)
\end{equation}
always holds.

The determination of the value of $c_0(f)$ has been studied deeply.
In the seminal work of Verchenko \cite{var76}, when $f$ is real analytic and {\it nondegenerate with respect to its Newton polyhedron},
the index $c_0(f)$ is represented by using the {\it Newton distance} $d(f)$ of $f$ as 
\begin{equation}\label{eqn2:6}
c_0(f)=1/d(f)
\end{equation}
under some conditions. 
(See \cite{var76}, \cite{kn20} for the definitions of the Newton polyhedron,
the nondegeneracy with respect to the Newton polyhedron, the Newton distance, etc.)
We remark that if $f$ is real analytic, then $Z_f(\varphi)(s)$ can be meromorphically continued to the whole complex plane,
and in particular it has a pole at $s=-c_0(f)$ under the condition \eqref{eqn:2.1}.
In the same paper \cite{var76},
the two-dimensional case is also investigated deeply.
Indeed, it is shown that for real analytic $f$,
the equality \eqref{eqn2:6} holds  in {\it adapted coordinates}.
In smooth cases, similar results are given in \cite{gre06}, \cite{kn16tor}.
In particular, Greenblatt \cite{gre06} obtains a sharp result in the two-dimensional case.
\begin{theorem}[\cite{gre06}]\label{thm:2.0}
When $f$ is a nonflat smooth function defined near the origin in $\R^2$,
the equation $c_0(f)=1/d(f)$ holds in adapted coordinates.
\end{theorem}

We refer to the paper \cite{im11tams} by Ikromov and M\"uller, which gives equivalent conditions for adaptedness of smooth coordinates
in two dimensions.

\begin{remark}\label{rem:2.2}
In this paper, we focus on functions $f$ of the form \eqref{eqn:1.5}.
In this case, it follows from the results in \cite{im11tams}
that $f$ is expressed in an adapted coordinate
and $d(f)=\max \{a,b\}$.
Then Theorem \ref{thm:2.0} implies that
the integral $Z_f(\varphi)(s)$ becomes a holomorphic function on the half-plane ${\rm Re}(s)>-1/d(f)$ and
has a singularity at $s=-1/d(f)$
under the condition \eqref{eqn:2.1} on $\varphi$,
which implies
$h_0(f)=\min\{1/a,1/b\}$.
\end{remark}

\subsection{Meromorphic continuation of local zeta functions}\label{sec:2.2}
Let $f, \varphi$ be as in \eqref{eqn:1.1}.
In a similar fashion to the investigation of holomorphic continuation,
we consider the following quantities, which are introduced in \cite{kn20}:
\begin{equation}
\begin{split}
m_0(f,\varphi)&:=\sup\left\{\rho>0: 
\begin{array}{l} 
\mbox{The domain to which $Z_f(\varphi)$ can} \\ 
\mbox{be meromorphically continued} \\
\mbox{contains the half-plane ${\rm Re}(s)>-\rho$}
\end{array}
\right\},\\
m_0(f)&:=\inf\left\{m_0(f,\varphi):
\varphi\in C_0^{\infty}(U)\right\}.
\end{split}
\end{equation}
From the definitions,
we have
\begin{equation}\label{eqn:2.8}
\begin{split}
&h_0(f)\leq m_0(f) \leq m_0(f,\varphi),\\
&h_0(f,\varphi)\leq m_0(f,\varphi).\\
\end{split}
\end{equation}

As mentioned in the Introduction, meromorphic continuation of $Z_f(\varphi)(s)$ is precisely investigated in real analytic case.
In particular, if $f$ is real analytic, then $m_0(f)=\infty$ holds.
When $f$ is not real analytic,
on the other hand, it may occur in smooth cases that
$m_0(f)$ is finite.
Indeed, the following specific smooth functions, which are not real analytic, are investigated in \cite{kn19}:
\begin{equation}\label{eqn:2.2}
f(x,y)=x^a y^b+x^a y^{b-q} e^{-1/|x|^p},
\end{equation}
where $a,b,p,q$ satisfy that
\begin{itemize}
\item
$a,b$ are nonnegative integers satisfying $a<b, 2\leq b$,
\item
$p$ is a positive real number,
\item
$q$ is an even integer satisfying $2\leq q \leq b$.
\end{itemize}
The functions \eqref{eqn:2.2} are special cases of \eqref{eqn:1.5}.
From Remark \ref{rem:2.2},
the integral $Z_f(\varphi)(s)$ becomes a holomorphic function on the half-plane ${\rm Re}(s)>-1/b$ and
has a singularity at $s=-1/b$
under the condition \eqref{eqn:2.1} on $\varphi$.
More precisely, it is shown in \cite{kn19} that
the singularity of $Z_f(\varphi)(s)$ at $s=-1/b$
is different from poles. (See Remark \ref{rem:3.6} below.)
In particular, 
the value of $m_0(f)$ is determined as follows.

\begin{theorem}[\cite{kn19}]\label{thm:2.1}
Let $f$ be as in \eqref{eqn:2.2}.
Then $m_0(f)=1/b$ holds.
\end{theorem}

In general, it is difficult to determine the values of $m_0(f)$ for smooth functions $f$.

\subsection{A classification of smooth model cases}

In this section, we consider smooth functions $f$ of the form \eqref{eqn:1.5},
which are regarded as a model of smooth functions in two dimensions.
Without loss of generality, we assume that $a\leq b$ in \eqref{eqn:1.5}.
From Remark \ref{rem:2.2} and \eqref{eqn:2.8},
we have the equality: $h_0(f)=1/b$ and a uniform lower estimate of $m_0(f)$ as $m_0(f)\geq 1/b$.

Now, we have a classification of functions of the form \eqref{eqn:1.5} by using Taylor's formula.

\begin{lemma}[Lemma 3.1 of \cite{kn20}]\label{lem:2.1}
Let $f$ be as in \eqref{eqn:1.5}.
If $U$ is sufficiently small, then
$f$ can be expressed on $U$ 
as one of the following four forms. 
\begin{enumerate}
\item[(A)]
$f(x,y)=
v(x,y)x^a y^b$,
\item[(B)]
$f(x,y)=
v(x,y)x^a y^b 
+ g(x,y)$,
\item[(C)]
$f(x,y)=
v(x,y)x^a y^b 
+ h(x,y)$,
\item[(D)]
$f(x,y)=
v(x,y)x^a y^b 
+ g(x,y)+h(x,y)$,
\end{enumerate}
where $a\leq b$ and $v,g,h$ are smooth functions defined on $U$ 
satisfying the following properties:
\begin{itemize}
\item $v(0,0)=u(0,0)\neq 0$.
\item 
$g, h$ are {\it non-zero} flat functions admitting the forms:
\begin{equation}
g(x,y)=\sum_{j=0}^{b-1}y^{j}g_j(x)
\quad
\mbox{ and } \quad
h(x,y)=\sum_{j=0}^{a-1}x^{j}h_j(y),
\end{equation}
where $g_j, h_j$ are flat at the origin.
\end{itemize}
\end{lemma}

\begin{remark}\label{rem:2.5}
Let us consider the above four cases under the following special conditions.
\begin{enumerate}
\item When $a=b$, the cases (B) and (C) are equivalent under switching  the $x$ and $y$ variables.
\item The case (C) with the condition: $a=0$ belongs to the case (A).
\item The case (D) with the condition: $a=0$ belongs to the case (B).
\item Under the condition: $b=0$, which implies $a=0$,
the cases (B), (C), (D) belong to the case (A).
\item
The case (B) with the condition: $b=1$ is reduced to the case (A)
under some coordinate change.
\item
The case (C) with the condition: $a=1$ is reduced to the case (A)
under some coordinate change.
\item
The case (D) with the condition: $a=b=1$ is reduced to the case (A)
under some coordinate change by using the Morse lemma.
\end{enumerate}
\end{remark}

When $f$ belongs to the case (A),
$f$ can be expressed as a monomial under a smooth coordinate change on $\R^2$,
which gives that $m_0(f)=\infty$ holds.
In the case (C), we have a nontrivial lower estimate of $m_0(f)$
as follows.

\begin{theorem}[\cite{kn20}]\label{thm:2.2}
Let $f$ be as in $(C)$ in Lemma $\ref{lem:2.1}$.
Then $m_0(f)\geq 1/a$ holds.
\end{theorem}

A list of the above mentioned results is given in Table 1.
\begin{table}[H]
\caption[]{The values of
$h_0(f)$ and 
$m_0(f)$.}
\begin{center}
\begin{tabular}{r|cccc} \toprule
 & (A) & (B) & (C) & (D) \\  \midrule 
$h_0(f)$ & $1/b$ & $1/b$ & $1/b$ & $1/b$ \\
$m_0(f)$  & $\infty$ & $\geq 1/b$ & $\geq 1/a$ & $\geq 1/b$ 
\\ \bottomrule
\end{tabular}
\end{center}
\end{table}
Here,
the following question naturally arises.
\begin{question}
Are the lower estimates of $m_0(f)$ in Table 1 optimal ?
Here, their optimalities mean that
there exists a smooth function
$f$ of each form such that the value of $m_0(f)$
attains the lower bound of $m_0(f)$.
\end{question}

Since $f$ in \eqref{eqn:2.2} belongs to the case (B),
Theorem \ref{thm:2.1} gives the optimality of the estimate:
$m_0(f) \geq 1/b$ in the case (B).
In this paper, we consider the cases (C) and (D) in Question 1.

In \cite{kam22}, Kamimoto introduces a certain quantity for a non-flat smooth function $f$ and obtains a general lower estimate of $m_0(f)$ by using this quantity.
His result implies the lower estimates in Table 1.

\section{Main results}

Let us consider the integrals $Z_f(\varphi)(s)$ with 
smooth functions of the following form:
\begin{equation}\label{eqn:3.1}
f(x,y)=x^ay^b+x^{a-q}y^{b}e^{-1/|y|^p},
\end{equation}
where $a,b,p,q$ satisfy that
\begin{itemize}
\item
$a,b$ are positive integers satisfying
$2\leq a\leq b$,
\item
$p$ is a positive real number,
\item
$q$ is an even integer satisfying $2\leq q\leq a$.
\end{itemize}
We remark that $e^{-1/|y|^p}$ is regarded as a smooth function defined on $\R$ by considering that its value at $0$ takes $0$.
The function $e^{-1/|y|^p}$ is flat at the origin and 
is not real analytic.
(This function also appears in \eqref{eqn:2.2} and \eqref{eqn:3.7} below.)
By noticing that $f$ belongs to the case (C) in Lemma \ref{lem:2.1},
Theorem \ref{thm:2.2} implies the existence of meromorphic continuation of $Z_f(\varphi)(s)$
to the half-plane ${\rm Re}(s)>-1/a$.
We obtain the asymptotic limit of local zeta functions associated with $f$ as $s\to -1/a+0$ along the real axis.

\begin{theorem}\label{thm:3.1}
Let $f$ be as in \eqref{eqn:3.1}
and $\varphi$ as in \eqref{eqn:1.1}.
Then $Z_{f}(\varphi)(s)$ can be analytically continued as a meromorphic function, denoted again by $Z_f(\varphi)(s)$, to the half-plane ${\rm Re}(s)>-1/a$,
and its poles are contained in the set $\{-j/b:j\in \N \mbox{ with } j<b/a\}$.
Moreover the following hold:
\begin{enumerate}
\item
If at least one of the following two conditions is satisfied:
\begin{enumerate}
\item
$0<p<b/a-1$,
\item
$b/a$ is not an odd integer,
\end{enumerate}
then
\begin{equation}\label{eqn:3.2}
\lim_{\sigma\to -1/a+0}(a\sigma+1)^{1+\frac{b/a-1}{p}}
\cdot Z_f(\varphi)(\sigma) =
-4C\cdot \varphi(0,0),
\end{equation}
where $C$ is the positive constant defined by
\begin{equation}\label{eqn:3.3}
C=\frac{q^{(b/a-1)/p}}{(b/a-1)}\Gamma\left(1+\frac{b/a-1}{p}\right).
\end{equation}
Here, $\Gamma$ denotes the gamma function.
\item
If $p=b/a-1$ and $b/a$ is an odd integer,
then
\begin{equation}
\lim_{\sigma\to -1/a+0}(a\sigma+1)^{2}
\cdot Z_f(\varphi)(\sigma) =
-\frac{4q}{p}\cdot \varphi(0,0)
+\frac{4a}{b p!}\cdot \frac{\d^{p}\varphi}{\d y^{p}}(0,0).
\end{equation}
\item
If $p>b/a-1$ and $b/a$ is an odd integer,
then
\begin{equation}
\lim_{\sigma\to -1/a+0}(a\sigma+1)^{2}
\cdot Z_f(\varphi)(\sigma) =
\frac{4a}{b m!}\cdot \frac{\d^{m}\varphi}{\d y^{m}}(0,0),
\end{equation}
where 
$m=b/a-1$.
\end{enumerate}
\end{theorem}

We notice that
if $Z_f(\varphi)(s)$ can be meromorphically continued across the point $s=-1/a$
and 
it had a pole of order $n$ at $s=-1/a$,
then 
the limit 
$\lim_{\sigma\to -1/a+0}(a\sigma+1)^{n}\cdot Z_f(\varphi)(\sigma)$
must be a nonzero constant.
From Theorem \ref{thm:3.1} (i), we have the following.

\begin{corollary}\label{cor:3.2}
Let $f$ be as in \eqref{eqn:3.1} and $\varphi$ as in \eqref{eqn:1.1}.
Assume that at least one of the following two conditions is satisfied:
\begin{itemize}
\item[(a)]
$0<p<b/a-1$,
\item[(b)]
$b/a$ is not an odd integer.
\end{itemize}
If $(b/a-1)/p$ 
is not an integer and $\varphi(0,0)\neq 0$,
then $Z_f(\varphi)(s)$ cannot be meromorphically continued to any open neighborhood of $s=-1/a$ in $\C$.
In particular,
$m_0(f)=1/a$ holds.
\end{corollary}

\begin{remark}\label{rem:3.3}
When $a=b$ in Theorem \ref{thm:3.1},
which is contained in the case (iii),
$Z_f(\varphi)(s)$ can be analytically continued as a holomorphic function to the half-plane ${\rm Re}(s)>-1/a$ and
\begin{equation}
\lim_{\sigma \to -1/a+0}(a\sigma+1)^2\cdot Z_f(\varphi)(\sigma)=4\varphi(0,0).
\end{equation}
However, this information is not sufficient to see that
the singularity is different from poles at $s=-1/a$.
\end{remark}

We also consider the following smooth functions belonging to the case (D) in Lemma $\ref{lem:2.1}$:
\begin{equation}\label{eqn:3.7}
f(x,y)=x^ay^b+x^{a}y^{b-q}e^{-1/|x|^p}+x^{a-\tilde{q}}y^{b}e^{-1/|y|^{\tilde{p}}},
\end{equation}
where $a,b,p,\tilde{p},q,\tilde{q}$ satisfy that
\begin{itemize}
\item
$a,b$ are positive integers satisfying
$2\leq a< b$,
\item
$p$, $\tilde{p}$ are positive real numbers,
\item
$q$, $\tilde{q}$ are even integers satisfying 
$2\leq q\leq b$, $2\leq \tilde{q}\leq a$.
\end{itemize}
From Remark \ref{rem:2.2}, 
the integral $Z_f(\varphi)(s)$ becomes a holomorphic function
on the half-plane ${\rm Re}(s)>-1/b$
and has some singularity at $s=-1/b$
under the condition \eqref{eqn:2.1} on $\varphi$.

\begin{theorem}\label{thm:3.3}
Let $f$ be as in \eqref{eqn:3.7} and $\varphi$ as in \eqref{eqn:1.1}.
Then the following hold:
\begin{enumerate}
\item
If $p>1-a/b$, then
\begin{equation}\label{eqn:3.9}
\lim_{\sigma \to -1/b+0}(b\sigma+1)^{1-\frac{1-a/b}{p}}\cdot Z_f(\varphi)(\sigma)=4\hat{C}\cdot \varphi(0,0),
\end{equation}
where $\hat{C}$ is the positive constant defined by
\begin{equation}\label{eqn:3.10}
\hat{C}=\frac{1}{q^{{(1-a/b)/p}}(1-a/b)}\Gamma\left(1-\frac{1-a/b}{p}\right).
\end{equation}
Here, $\Gamma$ denotes the gamma function.
\item
If $p=1-a/b$, then
\begin{equation}
\lim_{\sigma \to -1/b+0}\left|\log(b\sigma+1)\right|^{-1}\cdot Z_f(\varphi)(\sigma)=\frac{4}{pq}\cdot \varphi(0,0).
\end{equation}
\item
If $0<p<1-a/b$, then there exists a constant $C(\varphi)$ which depends on $f,\varphi$ but is independent of $\sigma$ such that
\begin{equation}\label{eqn:3.12}
\lim_{\sigma \to -1/b+0}Z_f(\varphi)(\sigma)=C(\varphi),
\end{equation}
where $C(\varphi)$ is positive if $\varphi$ satisfies the condition \eqref{eqn:2.1}.
\end{enumerate}
\end{theorem}

\begin{remark}\label{rem:3.6}
When $f$ is as in \eqref{eqn:2.2},
the same assertions (i), (ii), (iii) in Theorem \ref{thm:3.3} hold. (See \cite{kn19}.)
Theorem \ref{thm:3.3} implies that
the last term in \eqref{eqn:3.7}
does not affect the asymptotic limits of $Z_f(\varphi)(\sigma)$
as $\sigma \to -1/b+0$
when $p\geq 1-a/b$.
Needless to say,
the constant $C(\varphi)$ in \eqref{eqn:3.12} depends on all terms of $f$.
\end{remark}

Theorem \ref{thm:3.3} gives a result
analogous to the case where $f$ is as in \eqref{eqn:3.1}.
(See also Remark \ref{rem:2.2}.)

\begin{corollary}\label{cor:3.4}
Let $f$ be as in \eqref{eqn:3.7} and $\varphi$ as in \eqref{eqn:1.1}
with the condition \eqref{eqn:2.1}.
Then $Z_f(\varphi)(s)$ cannot be meromorphically continued to any open neighborhood of $s=-1/b$ in $\C$.
In particular, $m_0(f)=1/b$ holds.
\end{corollary}

\begin{remark}
In Theorems \ref{thm:3.1} and \ref{thm:3.3},
the divergence rates of $Z_f(\varphi)(\sigma)$ as $\sigma\to -1/a+0$
and $\sigma\to -1/b+0$ are seen 
respectively.
In particular, the exponent of the divergence rate is greater than two
in the case (i-a) in Theorem \ref{thm:3.1},
which seems strange and interesting since the situation is different from that of the real analytic case in two dimensions.
Indeed, if $f$ is real analytic in $\R^2$, then $Z_f(\varphi)(s)$ has poles of order at most two. 
We would point out that the asymptotic limits \eqref{eqn:3.2} and \eqref{eqn:3.9} 
with respect to the functions \eqref{eqn:3.1} and \eqref{eqn:2.2}
have the same form in some sense.
(See Remark \ref{rem:3.6}.)
In fact, we see the following.
\begin{itemize}
\item
By switching the variables $x$, $y$ and the exponents $a$, $b$ in \eqref{eqn:3.1},
the difference between the functions \eqref{eqn:3.1} and \eqref{eqn:2.2} is regarded as that between the conditions: $a\geq b$ and $a<b$.
\item
The equation \eqref{eqn:3.2} coincides with the equation \eqref{eqn:3.9} under switching $a$ and $b$.
\end{itemize}
In other words, Theorem \ref{thm:3.1} (i) implies that
the equation \eqref{eqn:3.9} with respect to the function \eqref{eqn:2.2} also holds
when $a\geq b$ and $0<p<a/b-1$ (or $a/b$ is not an odd integer).
If $0<p<a/b-1$, of course, then the exponent $1-(1-a/b)/p$ in \eqref{eqn:3.9} is greater than two.
\end{remark}

Now, we see from Corollaries \ref{cor:3.2} and \ref{cor:3.4} that
there exists a smooth function
$f$ belonging to each of the cases $(C)$ and $(D)$ in Lemma $\ref{lem:2.1}$ such that the value of $m_0(f)$
attains the lower bound of $m_0(f)$ in Table 1.
Therefore we obtain an answer to Question 1.

\begin{theorem}
Let $f$ be as in \eqref{eqn:1.5}.
Then the following hold.
\begin{enumerate}
\item
If $f$ belongs to the case $(C)$, then the estimate: $m_0(f)\geq 1/a$
is optimal.
\item
If $f$ belongs to the case $(D)$, then the estimate: $m_0(f)\geq 1/b$
is optimal.
\end{enumerate}
Here, the cases $(C)$, $(D)$ are as in Lemma $\ref{lem:2.1}$.
\end{theorem}
Consequently, it is shown that each lower estimate of $m_0(f)$ in Table 1 is optimal.
More detailed situations should be discussed.

\begin{question}
For {\it any} $a,b\in \N$ with $a\leq b$,
does there exist a smooth function $f$ of the form \eqref{eqn:1.5}
which belongs to each of the cases (B), (C), (D) in Lemma $\ref{lem:2.1}$ such that the value of $m_0(f)$ attains the lower bound of $m_0(f)$ in Table 1 ?
\end{question}

It is easy to see from Remark \ref{rem:2.5} that the condition:
$b\geq 2$ (resp. $a\geq 2$)
is necessary for each of the cases (B), (D) (resp. the case (C))
in the question.
Summarizing Theorem \ref{thm:2.1}, Corollaries \ref{cor:3.2} and \ref{cor:3.4},
we have a partial answer to Question 2. (See also Remark \ref{rem:3.3}.)

\begin{theorem}
Let $a,b$ be natural numbers with $a< b$.
\begin{enumerate}
\item
When $b\geq 2$, there exists a smooth function $f$ belonging to
the case $(B)$ such that $m_0(f)=1/b$.
\item
When $a\geq 2$, there exists a smooth function $f$ belonging to
the case $(C)$ such that $m_0(f)=1/a$.
\item
When $a\geq 2$, there exists a smooth function $f$ belonging to
the case $(D)$ such that $m_0(f)=1/b$.
\end{enumerate} 
Here, the cases $(B)$, $(C)$, $(D)$ are as in Lemma $\ref{lem:2.1}$.
\end{theorem}

We remark that the following cases remain still unsolved in Question 2.
(See also Remark \ref{rem:2.5}.)
\begin{itemize}
\item
$a=b$ in the cases (B), (C), (D),
\item
$a=1$, $b\geq 2$ in the case (D).
\end{itemize}

Moreover,
it is naturally expected that the value of $m_0(f)$ is determined by the optimal lower bound in each of the cases (B), (C), (D).
In other words, the following question arises.

\begin{question}
Do the following equalities hold ?
\begin{enumerate}
\item
For {\it all} $f$ belonging to the case $(B)$ or $(D)$ with $b\geq 2$, $m_0(f)=1/b$;
\item
For {\it all} $f$ belonging to the case $(C)$ with $a\geq 2$, $m_0(f)=1/a$.
\end{enumerate}
\end{question}

We obviously cannot answer Question 3 at present.
In particular, it would be difficult to apply our analysis in this paper to show the equalities in general.
In fact, 
a substantial part of our analysis is a precise investigation of certain integrals associated with the function $e^{-1/|u|^p}$
appearing in \eqref{eqn:3.1} and \eqref{eqn:3.7},
and the properties of the function $e^{-1/|u|^p}$ play important roles in our analysis.
(See Section \ref{sec:4.1} below.)

The properties of non-poler singularities of $Z_{f}(\varphi)(s)$ at $s=-m_{0}(f)$ should also be investigated.
The following question is elementary, but is still open.
\begin{question}
Are the non-poler singularities of $Z_f(\varphi)(s)$ at $s=-m_0(f)$ isolated or not ?
\end{question}
If these singularities are isolated, then they are essential singularities.
At present, it seems impossible to answer the question with the information given by the above results only.

\section{Auxiliary lemmas}
\subsection{Asymptotics of certain integrals}\label{sec:4.1}

For $p>0$ and $q\in \N$,
let $e_{p,q}(u)$ be the smooth function on $\R_+$, denoted by $e(u)$ for simplicity,
defined by
\begin{equation}\label{eqn:4.1}
\begin{split}
e(u)=\exp\left(\frac{-1}{qu^p} \right)
\end{split}
\end{equation}
for $u>0$ and $e(0)=0$.
The function $e$ is flat at the origin.

\begin{lemma}\label{lem:4.1}
For $X>0$, let
\begin{equation}\label{eqn:4.2.0}
K(X)=\int_0^{r}u^{A+BX}(e(u))^{X}du,
\end{equation}
where $r>0$, $A, B \in \R$ and $r,A,B$ are independent of $X$.
Note that the integral converges for all $X$.
Then the following hold.
\begin{enumerate}
\item
If $A<-1$ and $(A+1)/p$ is not an integer, then
\begin{equation}\label{eqn:4.3.0}
\begin{split}
K(X)
=&\sum_{m=0}^{N-1}\sum_{l=0}^{m}a_{m,l}X^{(A+1)/p+m}(\log X)^{m-l}
+\frac{r^{A+1}}{A+1}\\
&+\sum_{l=0}^{N}a_{N,l}X^{(A+1)/p+N}(\log X)^{N-l}
+O(X)\quad \mbox{as $X\to +0$},
\end{split}
\end{equation}
where $N=\max\{m\in \N:m<-(A+1)/p+1\}$,
\begin{equation}\label{eqn:4.4.0}
a_{m,l}=\frac{(-1)^{l}B^{m}}{(m-l)!p^{m+1}q^{(A+1)/p}}\sum_{k=0}^{l}\frac{(\log q)^{l-k}\Gamma^{(k)}(-(A+1)/p)}{(l-k)!k!}.
\end{equation}
Here, $\Gamma^{(k)}$ denotes the $k$-th derivative of the gamma function.
\item
If $A<-1$ and $(A+1)/p$ is an integer, then
\begin{equation}
\begin{split}
K(X)
=&\sum_{m=0}^{N-1}\sum_{l=0}^{m}a_{m,l}X^{(A+1)/p+m}(\log X)^{m-l}
+\sum_{l=0}^{N-1}a_{N,l}(\log X)^{N-l}\\
&+a_{N,N}+\frac{r^{A+1}}{A+1}
+\sum_{l=0}^{N}a_{(N+1),l}X(\log X)^{N-l+1}
+O(X)\quad \mbox{as $X\to +0$},
\end{split}
\end{equation}
where $N=-(A+1)/p$ and $a_{m,l}$ is as in \eqref{eqn:4.4.0}.
\item
If $A=-1$,
then
\begin{equation}\label{eqn:4.6}
\begin{split}
K(X)=&-\frac{1}{p}\log X-\frac{\gamma}{p}+\frac{\log q}{p}+\log r-\frac{B}{2p^2}X(\log X)^2\\
&-\frac{B\gamma}{p^2}X(\log X)+O(X)
\quad \mbox{as $X\to +0$},
\end{split}
\end{equation}
where
$\gamma$ is Euler's constant.
\item
If $A>-1$,
then
\begin{equation}
\lim_{X\to +0}K(X)
=\frac{r^{A+1}}{A+1}.
\end{equation}
\end{enumerate}
\end{lemma}

\begin{remark}\label{rem:4.2}
From the lemma, we have that if $A<-1$, then
\begin{equation}
\lim_{X \to +0}X^{-(A+1)/p}\cdot K(X)=\frac{\Gamma(-(A+1)/p)}{p q^{(A+1)/p}}(=a_{0,0}),
\end{equation}
where $a_{m,l}$ is as in \eqref{eqn:4.4.0}.
Precisely, the following hold.
\begin{enumerate}
\item
If $A< -p-1$, which implies that $(A+1)/p<-1$,
then
\begin{equation}\label{eqn:4.9}
\lim_{X\to +0}X^{-(A+1)/p-1}(\log X)^{-1}\cdot\left(K(X)-a_{0,0} X^{(A+1)/p}\right)
=\frac{B\Gamma(-(A+1)/p)}{p^2q^{(A+1)/p}}
(= a_{1,0}).
\end{equation}
\item
If $A= -p-1$, which implies that $(A+1)/p=-1$,
then
\begin{equation}
\lim_{X\to +0}(\log X)^{-1}\cdot\left(K(X)-a_{0,0} X^{-1}\right)
=\frac{B\Gamma(-(A+1)/p)}{p^2q^{(A+1)/p}}
(= a_{1,0}).
\end{equation}
Hence, the equation \eqref{eqn:4.9} also holds in the case where $A=-p-1$.
\item
If $-p-1<A<-1$, which implies that $-1<(A+1)/p<0$,
then
\begin{equation}
\lim_{X\to +0}X^{-(A+1)/p-1}(\log X)^{-1}\cdot\left(K(X)-a_{0,0} X^{(A+1)/p}-\frac{r^{A+1}}{A+1}\right)
= \frac{B\Gamma(-(A+1)/p)}{p^2q^{(A+1)/p}}
(= a_{1,0}).
\end{equation}
\end{enumerate}
\end{remark}

The following lemma is an essential part of the proof of Lemma \ref{lem:4.1}.
\begin{lemma}\label{lem:4.3}
For $X>0$, let
\begin{equation}\label{eqn:4.2.2}
L(X)=\int_{cX}^{\infty}t^{a+bX-1}e^{-t}dt,
\end{equation}
where $a\geq 0$, $b\in \R$, $c>0$ and $a,b,c$ are independent of $X$.
Note that the integral converges for all $X$.
Then the following hold.
\begin{enumerate}
\item
If $a>0$, then
\begin{equation}
\begin{split}
L(X)=&\sum_{m=0}^{N}\frac{b^m\Gamma^{(m)}(a)}{m!}X^{m}
-\frac{c^{a}}{a}X^{a+bX}+O(X^{a+1})
\quad \mbox{as $X\to +0$},
\end{split}
\end{equation}
where $N=\max\{m\in\N:m<a+1\}$ and $\Gamma^{(m)}$ is the $m$-th derivative of the gamma function.
\item
If $a=0$, then
\begin{equation}\label{eqn:4.14}
L(X)=-\log X-\log c-\gamma
-\frac{b}{2}X(\log X)^2-b(\log c)X(\log X)
+ O(X)
\quad\mbox{as $X\to +0$},
\end{equation}
where $\gamma$ is Euler's constant.
\end{enumerate}
\end{lemma}

\begin{remark}
(i)
The integral $L(X)$ is expressed as
\begin{equation}
L(X)=\Gamma(a+bX,cX),
\end{equation}
where $\Gamma(\a,x)=\int_{x}^{\infty}t^{\a-1}e^{-t}dt$.
It is known that
{\it the upper incomplete gamma function}
is defined by $\Gamma(\a,x)$
for $\a\in \C$ and $x\in \C$ with $-\pi<{\rm arg}x<\pi$.

(ii)
Since $X^{bX}=1+bX\log X+O(X^2(\log X)^2)$ as $X\to +0$,
we have more precise asymptotic behavior of $L(X)$ in the case where $a>0$ as follows.
\begin{equation}
\begin{split}
L(X)=&\sum_{m=0}^{N-1}\frac{b^m\Gamma^{(m)}(a)}{m!}X^{m}
-\frac{c^{a}}{a}X^{a}
+\frac{b^N\Gamma^{(N)}(a)}{N!}X^{N}\\
&-\frac{b c^{a}}{a}X^{a+1}\log X+O(X^{a+1})
\quad \mbox{as $X\to +0$}.
\end{split}
\end{equation}
\end{remark}

\begin{proof}
(i)
Assume $a>0$.
Let $X$ be sufficiently small.
Then the inequality: $a+bX>0$ holds.
We have that $L(X)=\Gamma(a+bX)-\gamma(a+bX,cX)$,
where $\Gamma$ is the gamma function and $\gamma(\alpha,x)=\int_{0}^{x}t^{\a-1}e^{-t}dt$.
We remark that
{\it the lower incomplete gamma function}
is defined by $\gamma(\alpha,x)$
when $\a,x\in \C$, ${\rm Re}(\a)>0$ and $-\pi<{\rm arg}x<\pi$.
It is known
\begin{equation}
\gamma(\a,x)=x^{\a}\sum_{k=0}^{\infty}\frac{(-1)^k x^k}{(\a+k)k!}
\left(=\frac{x^{\a}}{\a}\,_{1}F_{1}\left(
\begin{array}{c}
\a\\
\a+1
\end{array}
;-x\right)\right),
\end{equation}
where $_{1}F_{1}$ is {\it the generalized hypergeometric function of order $(1,1)$}.
See for example \cite{tem15}.
Hence,
\begin{equation}
\gamma(a+bX,cX)=c^{a+bX}X^{a+bX}\sum_{k=0}^{\infty}\frac{(-1)^k c^k X^k}{(a+k+bX)k!}.
\end{equation}
Since $\lim_{X\to +0}X^{bX}=1$,
we particularly have
\begin{equation}\label{eqn:4.19}
\gamma(a+bX,cX)=\frac{c^{a}X^{a+bX}}{a}+O(X^{a+1})
\quad\mbox{as $X\to +0$}.
\end{equation}
Since $\Gamma(a+bX)$ can be regarded as a real analytic function with respect to $X$ on a small neighborhood of $X=0$,
the Taylor series of $\Gamma(a+bX)$ at $X=0$ and \eqref{eqn:4.19}
give the assertion.

(ii)
Let us consider the case where $a=0$.
By Taylor's formula,
\begin{equation}\label{eqn:4.11.1}
t^{bX}=1+bX\log t+\varphi_1(t,X),
\end{equation}
where 
\begin{equation}\label{eqn:4.11.1.0}
\varphi_1(t,X)=b^{2}X^{2}(\log t)^{2}
\int_{0}^{1}(1-v)t^{bvX}dv.
\end{equation}
Note that
$\varphi_1(\cdot,X)$ is a smooth function
defined on $(cX, \infty)$ for each $X$.
We easily see that if $cX\leq t < 1$ or $b< 0$, then 
\begin{equation}\label{eqn:4.11.1.1}
|\varphi_1(t,X)|\leq C_1 X^{2}(\log t)^{2},
\end{equation}
where $C_1$ is some positive constant which is independent of $X$ and $t$.
On the other hand, if $t\geq 1$, $b\geq 0$ and $X<1$,
then 
\begin{equation}\label{eqn:4.11.1.2}
|\varphi_1(t,X)|\leq \tilde{C}_1 X^{2}t^{b}(\log t)^{2},
\end{equation} 
where $\tilde{C}_1$ is some positive constant which is independent of $X$ and $t$.
Substituting \eqref{eqn:4.11.1} to \eqref{eqn:4.2.2},
we have
\begin{equation}\label{eqn:4.24}
L(X)=E_1(cX)+bXL_1(X)+R(X),
\end{equation}
where $E_1(x)=\int_{x}^{\infty}t^{-1}e^{-t}dt$ is 
{\it the exponential integral},
\begin{equation}\label{eqn:4.11.2}
\begin{split}
L_1(X)&=\int_{cX}^{\infty}t^{-1}e^{-t}\log t dt,\\
R(X)&=\int_{cX}^{\infty}t^{-1}e^{-t}\varphi_1(t,X)dt.
\end{split}
\end{equation}
It is known that
\begin{equation}
E_1(x)=-\gamma-\log x-\sum_{k=1}^{\infty}\frac{(-1)^kx^k}{kk!},
\end{equation}
where the series converges for $x\in \R$, and $\gamma$ is Euler's constant. 
(See \cite{tem15} for example.)
Hence, we have
\begin{equation}\label{eqn:4.27}
E_1(cX)=-\log X-\log c-\gamma+ O(X)
\quad\mbox{as $X\to +0$}.
\end{equation}

Let us consider the integral $L_1(X)$ in \eqref{eqn:4.11.2}.
We decompose $L_1(X)$ as follows:
\begin{equation}
L_1(X)=\int_{cX}^{1}t^{-1}\log t dt+\int_{cX}^{1}t^{-1}\left(e^{-t}-1\right)\log t dt+\int_{1}^{\infty}t^{-1}e^{-t}\log t dt.
\end{equation}
Note that the last integral converges.
A simple computation gives
\begin{equation}
\int_{cX}^{1}t^{-1}\log t dt=-\frac{1}{2}\left(\log X+\log c\right)^2.
\end{equation}
On the other hand,
we have
\begin{equation}
\left|\int_{cX}^{1}t^{-1}\left(e^{-t}-1\right)\log t dt\right|
\leq \int_{cX}^{1}\left(-\log t\right) dt
=1+cX\log (cX) -cX.
\end{equation}
Hence
\begin{equation}\label{eqn:4.31}
L_1(X)=-\frac{1}{2}(\log X)^2-(\log c)(\log X)+O(1)
\quad \mbox{as $X\to +0$}.
\end{equation}

Lastly we consider the integral $R(X)$ in \eqref{eqn:4.11.2}.
When $b< 0$, we see from \eqref{eqn:4.11.1.1} that
\begin{equation}
\begin{split}
|R(X)|
&\leq C_1 X^{2} \int_{cX}^{\infty}t^{-1}e^{-t}(\log t)^2 dt\\
&=C_1 X^{2} \int_{cX}^{1}t^{-1}e^{-t}(\log t)^2 dv
+ C_1 X^{2} \int_{1}^{\infty}t^{-1}e^{-t}(\log t)^2 dt,
\end{split}
\end{equation}
where $C_1$ is as in \eqref{eqn:4.11.1.1}.
Since 
\begin{equation}\label{eqn:4.11.3}
\begin{split}
\int_{cX}^{1}t^{-1}e^{-t}(\log t)^2 dv
&\leq \int_{cX}^{1}t^{-1}(\log t)^2 dv\\
&=-\frac{1}{3}(\log X+\log c)^3
\end{split}
\end{equation}
and
the integral $\int_{1}^{\infty}t^{-1}e^{-t}(\log t)^2 dt$ converges,
we have
\begin{equation}\label{eqn:4.34}
R(X)=O(X^2|\log X|^3)\quad \mbox{as $X\to +0$}.
\end{equation}
In the case where $b\geq 0$, on the other hand, we have from \eqref{eqn:4.11.1.1} and \eqref{eqn:4.11.1.2} that
\begin{equation}\label{eqn:4.11.5}
|R(X)|\leq C_1 X^{2}\int_{cX}^{1}t^{-1}e^{-t}(\log t)^{2} dt
+\tilde{C}_1 X^{2}\int_{1}^{\infty}t^{-1+b}e^{-t}(\log t)^{2} dt,
\end{equation}
where $C_1$, $\tilde{C}_1$ are as in \eqref{eqn:4.11.1.1} and \eqref{eqn:4.11.1.2} respectively.
Note that the last integral converges.
From \eqref{eqn:4.11.3} and \eqref{eqn:4.11.5},
we also obtain \eqref{eqn:4.34}.

A simple computation gives \eqref{eqn:4.14} from \eqref{eqn:4.24}, \eqref{eqn:4.27}, \eqref{eqn:4.31}, \eqref{eqn:4.34}.
\end{proof}

\begin{proof}[Proof of Lemma \ref{lem:4.1}]
(i) Assume that $A<-1$ and $(A+1)/p$ is not an integer.
By changing the integral variable in \eqref{eqn:4.2.0}
by $t=X/(qu^p)$,
we have
\begin{equation}\label{eqn:4.36}
K(X)=\frac{X^{(A+1)/p+BX/p}}{p q^{(A+1)/p+BX/p}}\int_{X/(qr^p)}^{\infty}t^{-(A+1)/p-BX/p-1}e^{-t} dt.
\end{equation}
Applying Lemma \ref{lem:4.3} to the last integral
and noticing
\begin{equation}\label{eqn:4.37}
\begin{split}
q^{-BX/p}&=\sum_{k=0}^{\infty}\frac{(-1)^kB^k(\log q)^kX^k}{k!p^k},\\
X^{BX/p}&=\sum_{l=0}^{\infty}\frac{B^l X^l(\log X)^l}{l!p^l}
\end{split}
\end{equation}
for $X>0$,
we have
\begin{equation}
\begin{split}
K(X)
=&\frac{X^{(A+1)/p+BX/p}}{p q^{(A+1)/p+BX/p}}\sum_{m=0}^{N}\frac{(-1)^mB^m\Gamma^{(m)}(-(A+1)/p)}{m!p^m}X^{m}
+\frac{r^{A+1}}{A+1}+O(X)\\
=&\sum_{m=0}^{N}\sum_{l=0}^{m}a_{m,l}X^{(A+1)/p+m}(\log X)^{m-l}
+\frac{r^{A+1}}{A+1}+O(X)
\quad \mbox{as $X\to +0$},
\end{split}
\end{equation}
where $N=\max\{m\in\N:m<-(A+1)/p+1\}$
and $a_{m,l}$ is as in \eqref{eqn:4.4.0}.
Noticing $0<(A+1)/p+N<1$,
we have the truncated asymptotic expansion \eqref{eqn:4.3.0} of $K(X)$.

(ii)
In a similar fashion to the case (i),
we obtain from \eqref{eqn:4.36}, \eqref{eqn:4.37} and Lemma \ref{lem:4.3} that
\begin{equation}
\begin{split}
K(X)
=&\sum_{m=0}^{N}\sum_{l=0}^{m}a_{m,l}X^{(A+1)/p+m}(\log X)^{m-l}
+\sum_{l=0}^{N}a_{(N+1),l}X(\log X)^{N-l+1}\\
&+\frac{r^{A+1}}{A+1}+O(X)
\quad \mbox{as $X\to +0$},
\end{split}
\end{equation}
where $N=-(A+1)/p$ and $a_{m,l}$ is as in \eqref{eqn:4.4.0}.
Since $(A+1)/p+N=0$, we have the assertion.

(iii) We consider the case where $A=-1$.
Letting $A=-1$ in \eqref{eqn:4.36},
we have
\begin{equation}\label{eqn:4.40}
K(X)=\frac{X^{BX/p}}{p q^{BX/p}}\int_{X/(qr^p)}^{\infty}t^{-BX/p-1}e^{-t} dt.
\end{equation}
Applying Lemma \ref{lem:4.3} to the last integral,
we have
\begin{equation}\label{eqn:4.41}
\begin{split}
\int_{X/(qr^p)}^{\infty}t^{-BX/p-1}e^{-t} dt
=&-\log X-\gamma+\log q+p\log r+\frac{B}{2p}X(\log X)^2\\
&-\frac{B(\log q+p\log r)}{p}X(\log X)+O(X)
\quad \mbox{as $X\to +0$},
\end{split}
\end{equation}
where $\gamma$ is Euler's constant.
A simple computation gives \eqref{eqn:4.6} from \eqref{eqn:4.37}, \eqref{eqn:4.40}, \eqref{eqn:4.41}.

(iv)
We assume $A>-1$.
For sufficiently small $X>0$,
the inequality: $A+BX>-1$ holds.
Lebesgue's convergence theorem gives
\begin{equation}
\begin{split}
\lim_{X\to +0}L(X)
&=\int_{0}^{r}u^{A}du\\
&=\frac{r^{A+1}}{A+1}.
\end{split}
\end{equation}
\end{proof}

The following lemma which is essentially shown in \cite{kn19} 
will be used in the proof of Theorem \ref{thm:3.3}.
\begin{lemma}\label{lem:4.4}
For $X>0$, let
\begin{equation}\label{eqn:4.47.0}
\tilde{K}(X)=
\int_{0}^{r}u^{A+BX}(1-(e(u))^X)du,
\end{equation}
where $r>0$, $A,B\in\R$ and $r,A,B$ are independent of $X$.
Note that the integral converges if $A+BX>-1$.
Then the following hold.
\begin{enumerate}
\item
If $-1<A<p-1$,
then
\begin{equation}
\lim_{X\to +0}X^{-(A+1)/p}\cdot \tilde{K}(X)=\frac{\Gamma(-(A+1)/p+1)}{q^{(A+1)/p}(A+1)}.
\end{equation}
\item
If $A=p-1$, then
\begin{equation}
\lim_{X\to +0}X^{-1}|\log X|^{-1}\cdot \tilde{K}(X)
=\frac{1}{pq}.
\end{equation}
\end{enumerate}
\end{lemma}

\begin{remark}
We can see that if $A>p-1$,
then
\begin{equation}
\lim_{X\to +0}X^{-1}\cdot \tilde{K}(X)
=\frac{r^{A-p+1}}{q(A-p+1)}.
\end{equation}
Since the integral $\int_0^r u^{A+BX}du$
tends to $r^{A+1}/(A+1)$ as $X\to +0$
in the case where $A>-1$,
Lemma \ref{lem:4.4} gives some improvement of the assertion (iv) in Lemma \ref{lem:4.1}.
\end{remark}

To prove the lemma, we show the following.

\begin{lemma}\label{lem:4.7}
For $X>0$, let
\begin{equation}
\tilde{L}(X)=
\int_{cX}^{\infty}t^{a+bX-1}\left(1-e^{-t}\right)dt,
\end{equation}
where $a<0$, $b\in\R$, $c>0$ and $a,b,c$ are independent of $X$.
Note that the integral converges if $a+bX<0$.
Then the following hold.
\begin{enumerate}
\item
If $-1<a<0$,
then
\begin{equation}
\lim_{X\to +0}\tilde{L}(X)=-\frac{\Gamma(a+1)}{a}.
\end{equation}
\item
If $a=-1$, then
\begin{equation}
\lim_{X\to +0}|\log X|^{-1}\cdot\tilde{L}(X)=1.
\end{equation}
\end{enumerate}
\end{lemma}

\begin{proof}
(i)
We assume $-1<a<0$.
By integration by parts,
\begin{equation}\label{eqn:4.50}
\tilde{L}(X)=-\frac{(cX)^{a+bX}(1-e^{{-cX}})}{a+bX}
-\frac{1}{a+bX}\int_{cX}^{\infty}t^{a+bX}e^{-t}dt.
\end{equation}
Applying Lemma \ref{lem:4.3} to the last integral,
we particularly have 
\begin{equation}\label{eqn:4.51}
\lim_{X\to +0}\int_{cX}^{\infty}t^{a+bX}e^{-t}dt=\Gamma(a+1).
\end{equation}
On the other hand, we see
\begin{equation}\label{eqn:4.52}
-\frac{(cX)^{a+bX}(1-e^{{-cX}})}{a+bX}
=O(X^{a+1})
\quad\mbox{as $X\to +0$}
\end{equation}
from noticing that $\lim_{X\to +0}X^{bX}=1$ and $1-e^{-cX}=cX+O(X^2)$ as $X\to +0$.
Combining \eqref{eqn:4.50}, \eqref{eqn:4.51}, \eqref{eqn:4.52}
gives the assertion.

(ii)
We consider the case where $a=-1$.
Letting $a=-1$ in \eqref{eqn:4.50},
we have
\begin{equation}\label{eqn:4.53}
\tilde{L}(X)=-\frac{(cX)^{bX-1}(1-e^{{-cX}})}{bX-1}
-\frac{1}{bX-1}\int_{cX}^{\infty}t^{bX-1}e^{-t}dt.
\end{equation}
Applying Lemma \ref{lem:4.3} to the last integral,
we particularly have
\begin{equation}\label{eqn:4.54}
\lim_{X\to +0}|\log X|^{-1}\cdot\int_{cX}^{\infty}t^{bX-1}e^{-t}dt
=1.
\end{equation}
From \eqref{eqn:4.52}, \eqref{eqn:4.53}, \eqref{eqn:4.54},
we have
\begin{equation}
\lim_{X\to +0}|\log X|^{-1}\cdot\tilde{L}(X)=1.
\end{equation}
\end{proof}

\begin{proof}[Proof of Lemma \ref{lem:4.4}]
(i) 
We assume $-1<A<p-1$.
Changing the integral variable in \eqref{eqn:4.47.0} by
$t=X/(qu^p)$, we have
\begin{equation}\label{eqn:4.56}
\tilde{K}(X)
=\frac{X^{(A+1)/p+BX/p}}{p q^{(A+1)/p+BX/p}}\int_{X/(qr^p)}^{\infty}t^{-(A+1)/p-BX/p-1}\left(1-e^{-t}\right)dt.
\end{equation}
Here, we remark that
$\lim_{X\to +0}X^{BX/p}=1$.
From applying Lemma \ref{lem:4.7} to the integral in \eqref{eqn:4.56},
\begin{equation}
\lim_{X\to +0}X^{-(A+1)/p}\cdot \tilde{K}(X)=\frac{\Gamma(-(A+1)/p+1)}{q^{(A+1)/p}(A+1)}.
\end{equation}

(ii)
We consider the case where $A=-1$.
Letting $A=p-1$ in \eqref{eqn:4.56},
we have
\begin{equation}
\tilde{K}(X)
=\frac{X^{BX/p+1}}{p q^{BX/p+1}}\int_{X/(qr^p)}^{\infty}t^{-BX/p-2}\left(1-e^{-t}\right)dt.
\end{equation}
From applying Lemma \ref{lem:4.7} to the last integral
and remarking that
$\lim_{X\to +0}X^{BX/p}=1$,
\begin{equation}
\lim_{X\to +0}X^{-1}|\log X|^{-1}\cdot \tilde{K}(X)
=\frac{1}{pq}.
\end{equation}
\end{proof}

\subsection{Monotone convergence-type lemmas}

The following lemma is clear, but is often used in our analysis.
\begin{lemma}\label{lem:4.5}
Let $U$ be a domain in $\R^n$, $du$ the Lebesgue measure on $\R^n$,
$\epsilon>0$ and
$F(u,X)$ a real-valued function defined on $U \times (0,\epsilon)$ satisfying that
\begin{enumerate}
\item
$F(u,X)$ is nonnegative and $F \not\equiv 0$;
\item
$F(\cdot,X)$ is integrable on $U$ for any $X\in (0,\epsilon)$;
\item
$F(u,\cdot)$ is monotone decreasing on $(0,\epsilon)$ for any $u\in U$.
\end{enumerate}
For $X>0$, we define
\begin{equation}
M(X):=\int_{U}F(u,X)du.
\end{equation}
If $\limsup_{X\to +0}M(X)<\infty$, 
then the limit $\lim_{X\to +0}M(X)$ exists.
\end{lemma}
\begin{proof}
From the assumptions,
$M(X)$ is nonnegative
and monotone decreasing on $(0,\epsilon)$. 
From the boundedness of $\limsup_{X\to +0}M(X)$,
the limit of $M(X)$ as $X\to +0$ exists.
\end{proof}

A slightly general lemma is obtained.

\begin{lemma}\label{lem:4.6}
Let $U$, $du$, $\epsilon$, $F$ are as in Lemma $\ref{lem:4.5}$.
For $X>0$, we define
\begin{equation}
M(X;\psi):=\int_{U}F(u,X)\psi(u)du,
\end{equation}
where $\psi$ is a real-valued smooth function defined on $\overline{U}$.
If
$
\limsup_{X\to +0}\int_{U}F(u,X)du<\infty,
$
then there exists a constant $C(\psi)$ which is independent of $X$, such that
\begin{equation}\label{eqn:4.73}
\lim_{X\to +0}M(X;\psi)=C(\psi).
\end{equation}
Moreover, $C(\psi)$ is positive if $\psi$ satisfies the condition:
\begin{equation}\label{eqn:4.74}
\psi(0)>0\quad \mbox{and} \quad \psi(u)\geq 0 \mbox{ on $U$}.
\end{equation}
\end{lemma}

\begin{remark}
From the proof of the lemma below,
we have the estimate
\begin{equation}
-\sup_{U} (\psi_{-})\cdot\limsup_{X\to +0}\int_{U}F(u,X)du \leq C(\psi)\leq \sup_{U} (\psi_{+})\cdot\limsup_{X\to +0}\int_{U}F(u,X)du,
\end{equation}
where 
\begin{equation}\label{eqn:4.76}
\psi_{\pm}(u)=\max\{\pm\psi(u),0\}.
\end{equation}
\end{remark}

\begin{proof}
Let $\psi_{\pm}$ be as in \eqref{eqn:4.76}.
Then $\psi(u)=\psi_+(u)-\psi_-(u)$ holds.
We have
\begin{equation}\label{eqn:4.77}
M(X;\psi)
=M(X;\psi_+)-M(X;\psi_-).
\end{equation}
From the inequalities:
\begin{equation}
\begin{split}
M(X;\psi_{\pm}) \leq \sup_{U} (\psi_{\pm}) \int_{U}F(u,X)du,
\end{split}
\end{equation}
the boundedness of $\limsup_{X\to +0}\int_{U}F(u,X)du$
implies
that of $\limsup_{X\to +0}M(X;\psi_{\pm})$.
From applying Lemma \ref{lem:4.5} to $M(X;\psi_{\pm})$,
there exist nonnegative constants $C(\psi_{\pm})$ such that
\begin{equation}\label{eqn:4.79}
\lim_{X \to +0}M(X;\psi_{\pm})=C(\psi_{\pm}).
\end{equation}
Let $C(\psi):=C(\psi_{+})-C(\psi_{-})$.
From \eqref{eqn:4.77} and \eqref{eqn:4.79},
we obtain \eqref{eqn:4.73}.
Note that when $\psi$ satisfies the condition \eqref{eqn:4.74},
$C(\psi)=C(\psi_+)$ and $C(\psi_-)=0$ hold, 
so $C(\psi)$ is positive.
\end{proof}

\section{Analytic continuation of associated integrals}\label{sec:5}

Let us consider the following integrals:
\begin{equation}\label{eqn:5.1}
\begin{split}
Z_{\a\b}(\psi)(s)=\int_{V}\left|x^a y^b +x^{a-q} y^b e^{-1/|y|^p}\right|^s x^{\a}y^{\b}\psi(x)dxdy \quad
\mbox{for $s\in \C$},
\end{split}
\end{equation}
where
$a,b,p,q$ are as in \eqref{eqn:3.1},
$\a,\b$ are nonnegative integers,
\begin{equation}\label{eqn:5.2}
V=\{(x,y)\in \R^2:0\leq x\leq r_1, 0\leq y\leq r_2\}
\end{equation}
with 
$r_1,r_2\in (0,1)$,
$r_2$ is sufficiently small so that
\begin{equation}\label{eqn:5.2.0}
r_2<\min\left\{\left(\frac{p}{p+1}\right)^{1/p},\left(\frac{p}{b(\a+1)/a+p-1}\right)^{1/p}\right\}
\end{equation}
and $\psi$ is a real-valued smooth function defined on $[0,r_1]$.
Since the integral $Z_{\a\b}(\psi)(s)$ converges
locally uniformly on the half-plane ${\rm Re}(s)>\max\{-\frac{\a+1}{a},-\frac{\b+1}{b}\}$,
Lebesgue's convergence theorem gives that
$Z_{\a\b}(\psi)(s)$ becomes a holomorphic function
there.
Indeed, the convergence of $Z_{\a\b}(\psi)(s)$ on the strip $\max\{-\frac{\a+1}{a},-\frac{\b+1}{b}\}<{\rm Re}(s)<0$
is shown by the inequalities:
\begin{equation}\label{eqn:5.3}
\begin{split}
\left|Z_{\a\b}(\psi)(s)\right|
&\leq \sup(|\psi|)\int_{V}\left|x^a y^b +x^{a-q} y^b e^{-1/|y|^p}\right|^{\sigma} x^{\a}y^{\b}dxdy\\
&\leq \sup(|\psi|)\int_{V}x^{a\sigma+\a} y^{b\sigma+\b}dxdy,\\
\end{split}
\end{equation}
and the convergence on the half-plane ${\rm Re}(s)\geq 0$ is trivial.
\begin{remark}\label{rem:5.1}
(i) In our analysis, the following
slightly general integrals are also considered:
\begin{equation}\label{eqn:5.4}
\hat{Z}_{\a\b}(\phi)(s)=\int_{V}\left|x^a y^b +x^{a-q} y^b e^{-1/|y|^p}\right|^s x^{\a}y^{\b}\phi(x,y)dxdy \quad
\mbox{for $s\in \C$},
\end{equation}
where $a,b,p,q,\a,\b,V$ are as in \eqref{eqn:5.1}
and $\phi$ is a smooth function defined on $V$.
An analogous estimate to \eqref{eqn:5.3} implies that
$\hat{Z}_{\a\b}(\phi)(s)$
becomes a holomorphic function
on the half-plane ${\rm Re}(s)>\max\{-\frac{\a+1}{a},-\frac{\b+1}{b}\}$.\\
(ii) If $\b \geq b(\a+1)/a-1$, then
the integrals
$Z_{\a \b}(\psi)(s)$ and $\hat{Z}_{\a\b}(\phi)(s)$
become
holomorphic functions
on the half-plane ${\rm Re}(s)>-(\a+1)/a$.
\end{remark}

The purpose of this section is to prove the following theorem,
which shows the analytic continuation of $Z_{\a\b}(\psi)(s)$
to the half-plane ${\rm Re}(s)>-(\a+1)/a$
even when $\a,\b\in\Z_+$ satisfy $\b<b(\a+1)/a-1$.

\begin{theorem}\label{thm:5.2}
If $\b<b(\a+1)/a-1$, then the integral $Z_{\a\b}(\psi)(s)$
can be analytically continued as a meromorphic function
to the half-plane ${\rm Re}(s)>-(\a+1)/a$,
which has only one pole at $s=-(\b+1)/b$ of order at most one.
\end{theorem}

\begin{remark}\label{rem:5.2}
The above meromorphic function defined on the half-plane ${\rm Re}(s)>-(\a+1)/a$ may have a removable singularity at $s=-(\b+1)/b$.
If $\psi(0)$ is positive and $\psi$ is nonnegative,
then the function has only one simple pole at $s=-(\b+1)/b$.
\end{remark}

\subsection{Decompositions of associated integrals}\label{sec:5.1}
For a positive real number $\delta$,
the smooth function $E(\cdot)$ on $\R_+$ is defined by
\begin{equation}
\begin{split}
E(y)&:=y^{-\delta}e(y)
\end{split}
\end{equation}
for $y>0$, and $E(0)=0$,
where $e(\cdot)$ is as in \eqref{eqn:4.1}.

\begin{remark}\label{rem:5.3}
We remark some properties of the functions $e(\cdot)$ and $E(\cdot)$.
\begin{enumerate}
\item
$e(y)< E(y)$ on $(0,1)$. 
\item
The function $e(\cdot)$ is flat at the origin and so is the function $E(\cdot)$.
\item
A simple computation gives
\begin{equation}
\begin{split}
\frac{dE}{dy}&=-\delta y^{-\delta-1}e(y)+\frac{p}{q}y^{-p-\delta-1}e(y).
\end{split}
\end{equation}
\item
The functions $e(\cdot)$ and $E(\cdot)$ are monotonously increasing on
$\R_+$ and $[0,(p/q\delta)^{1/p}]$ respectively.
\end{enumerate}
\end{remark}

When $\delta$ satisfies
\begin{equation}\label{eqn:5.6.1}
\delta<\frac{p}{qr_2^p},
\end{equation}
we can define the functions $\rho: \R_+\to [0,r_2]$
and $\tau: \R_+\to [0,r_2]$
by
\begin{equation}\label{eqn:5.7}
\begin{split}
\rho(x)&=
\begin{cases}
e^{-1}(x) & \mbox{if } 0\leq x<e(r_2),\\
r_2 & \mbox{if } x\geq e(r_2),
\end{cases}
\\
\tau(x)&=
\begin{cases}
E^{-1}(x) & \mbox{if } 0\leq x<E(r_2),\\
r_2 & \mbox{if } x\geq E(r_2).
\end{cases}
\end{split}
\end{equation}
We notice from Remark \ref{rem:5.3} (iv) that the condition \eqref{eqn:5.6.1} is needed in the definition of the function $\tau$.
Note that
$e^{-1}(x)=\left(-q \log x\right)^{-1/p}$ on $(0,1)$ and
$\tau(x)\leq \rho(x) \leq r_2$ on $\R_+$.

Let us decompose the integral $Z_{\a\b}(\psi)(s)$.
From \eqref{eqn:5.2.0},
a positive real number $\delta$ is taken so that
\begin{equation}\label{eqn:5.8}
\max\left\{\frac{p+1}{q},\frac{b(\a+1)/a+p-1}{q}\right\}<\delta<\frac{p}{qr_2^p}.
\end{equation}
We decompose the set $V$ into the following three sets:
\begin{equation}\label{eqn:5.9}
\begin{split}
U_{1}&=U_{1}^{(\delta)}:=\{(x,y)\in V: x^q\geq y^{-q\delta}e^{-1/|y|^p}\ (\Leftrightarrow x \geq E(y))\},\\
U_{2}&=U_{2}^{(\delta)}:=\{(x,y)\in V: e^{-1/|y|^p}\leq x^q< y^{-q\delta}e^{-1/|y|^p} \ (\Leftrightarrow e(y) \leq x < E(y))\},\\
U_{3}&:=\{(x,y)\in V: x^q < e^{-1/|y|^p}\ (\Leftrightarrow x < e(y))\}.
\end{split}
\end{equation}
The integral $Z_{\a\b}(\psi)(s)$ is expressed as
\begin{equation}\label{eqn:5.10}
Z_{\a\b}(\psi)(s)=\sum_{j=1}^3 Z_{\a\b}^{(j)}(\psi)(s)
\end{equation}
with
\begin{equation}
Z_{\a\b}^{(j)}(\psi)(s)=\int_{U_j}|x^ay^b+x^{a-q}y^{b}e^{-1/|y|^p}|^s x^{\a} y^{\b}\psi(x)dxdy
\quad
\mbox{for $j=1,2,3$.}
\end{equation}
In a similar fashion to the integral $Z_{\a\b}(\psi)(s)$,
the integrals
$Z_{\a\b}^{(j)}(\psi)(s)$ become holomorphic functions
on the half-plane ${\rm Re}(s)>\max\{-(\a+1)/a,-(\b+1)/b\}$.

\subsection{Meromorphic continuation of $Z_{\a\b}^{(1)}(\psi)(s)$}
Let us investigate meromorphic continuation of $Z_{\a\b}^{(1)}(\psi)(s)$.
The following lemma will play a useful role in our analysis below.

\begin{lemma}\label{lem:5.4}
Let $\epsilon$ be a positive real number and
let $\Psi: (U_{1}^{(\epsilon)}\setminus\{(0,0)\}) \times \C\to \C$ be defined by
\begin{equation}\label{eqn:5.12}
\Psi(x,y;s)=\left|1+\frac{e^{-1/|y|^p}}{x^q}\right|^{s},
\end{equation}
where $U_{1}^{(\cdot)}$ is as in \eqref{eqn:5.9}
and $p,q$ are as in \eqref{eqn:5.1}.
Then we have the following:
\begin{enumerate}
\item
$\Psi(\cdot;s)$ can be continuously extended to $U_{1}^{(\epsilon)}$ for all $s\in \C$.
\item
$\Psi(x,y;\cdot)$ is an entire function for all $(x,y)\in U_{1}^{(\epsilon)}$.
\end{enumerate}
Moreover,
if $\epsilon>(p+1)/q$,
then the following hold:
\begin{enumerate}
\item[(iii)]
$\frac{\d \Psi}{\d y}(\cdot;s)$ can be continuously extended to $U_{1}^{(\epsilon)}$ for all $s\in \C$.
\item[(iv)]
$\frac{\d \Psi}{\d y}(x,y;\cdot)$
is an entire function for all $(x,y)\in U_{1}^{(\epsilon)}$.
\end{enumerate}
\end{lemma}

\begin{proof}
For simplicity, we denote $g(x,y):=1+x^{-q}e^{-1/|y|^p}$.
Notice that $g$ is a smooth function on $U_{1}^{(\epsilon)}\setminus\{(0,0)\}$.
Since 
$1\leq g(x,y)\leq 1+ y^{q\epsilon}$ on $U_{1}^{(\epsilon)}\setminus\{(0,0)\}$,
we have
\begin{equation}\label{eqn:5.14}
\lim_{U_{1}^{(\epsilon)}\ni (x,y)\to (0,0)}g(x,y)=1.
\end{equation}
From \eqref{eqn:5.14},
we can see that $\Psi$ has the properties (i), (ii).

A simple computation gives
\begin{equation}\label{eqn:5.15}
\begin{split}
\frac{\d \Psi}{\d y}(x,y;s)&=psx^{-q}y^{-p-1}e^{-1/|y|^p}\Psi(x,y;s-1)
\end{split}
\end{equation}
on  $(U_{1}^{(\epsilon)}\setminus\{(0,0)\})\times \C$.
It suffices to show that
$\frac{\d \Psi}{\d y}(\cdot;s)$
can be continuously extended to the origin.
From the definition of the set $U_{1}^{(\epsilon)}$,
we have
\begin{equation}
\left|\frac{\d \Psi}{\d y}(x,y;s)\right|\leq
p|s|y^{q\epsilon-p-1}\left|\Psi(x,y;s-1)\right|
\end{equation}
for $(x,y;s)\in(U_{1}^{(\epsilon)}\setminus\{(0,0)\}) \times \C$.
We assume that $\epsilon>(p+1)/q$. Then
\begin{equation}
\lim_{U_{1}^{(\epsilon)}\ni (x,y)\to (0,0)}\frac{\d \Psi}{\d y}(x,y;s)=0
\quad\mbox{ for } s\in \C.
\end{equation}
Therefore we have the property (iii).
The properties (ii), (iii) and \eqref{eqn:5.15} give us the property (iv)
and complete the proof of the lemma.
\end{proof}

\begin{remark}
We can see that
for sufficiently large $\epsilon>0$,
the higher derivatives of $\Psi$ with respect to the variable $y$ have similar properties to (iii), (iv) in Lemma \ref{lem:5.4}.
\end{remark}

\begin{lemma}\label{lem:5.6}
If $\b<b(\a+1)/a-1$,
then the integral $Z_{\a\b}^{(1)}(\psi)(s)$
can be analytically continued as a meromorphic function
to the half-plane ${\rm Re}(s)>-(\a+1)/a$,
which has only one pole at $s=-(\b+1)/b$ of order at most one.
\end{lemma}

\begin{proof}
By decomposing the integral region $U_1$ 
into the following two sets:
\begin{equation}\label{eqn:5.19}
\{(x,y)\in U_1:x\leq E(\tau(r_1))\}, \quad
\{(x,y)\in U_1:x> E(\tau(r_1))\},
\end{equation}
we have
\begin{equation}
\begin{split}
Z_{\a\b}^{(1)}(\psi)(s)
=&\int_{0}^{E(\tau(r_1))}
\left(\int_0^{E^{-1}(x)}x^{as+\a}y^{bs+\b}
\Psi(x,y;s) \psi(x)dy\right)dx\\
&+\int_{E(\tau(r_1))}^{r_1}
\left(\int_0^{r_2}x^{as+\a}y^{bs+\b}
\Psi(x,y;s) \psi(x)dy\right)dx,
\end{split}
\end{equation}
where $\Psi$ is as in \eqref{eqn:5.12}.
We remark that if $r_1<E(r_2)$, then the second set in \eqref{eqn:5.19} is empty.
From \eqref{eqn:5.8}, Lemma \ref{lem:5.4} and \eqref{eqn:5.15},
integration by parts with respect to the variable $y$
gives
\begin{equation}
\begin{split}
Z_{\a\b}^{(1)}(\psi)(s)
=&\frac{1}{bs+\b+1}W_{\a\b}^{(1)}(\psi)(s)-\frac{ps}{bs+\b+1}W_{\a\b}^{(2)}(\psi)(s)\\
&+\frac{r_2^{bs+\b+1}}{bs+\b+1}W_{\a\b}^{(3)}(\psi)(s)-\frac{ps}{bs+\b+1}W_{\a\b}^{(4)}(\psi)(s)
\end{split}
\end{equation}
with
\begin{equation}\label{eqn:5.22}
\begin{split}
W_{\a\b}^{(1)}(\psi)(s)&=\int_{0}^{E(\tau(r_1))}x^{as+\a}(E^{-1}(x))^{bs+\b+1}
\Psi(x,E^{-1}(x);s)
\psi(x)dx,\\
W_{\a\b}^{(2)}(\psi)(s)&=\int_{0}^{E(\tau(r_1))}
\left(\int_0^{E^{-1}(x)}x^{as-q+\a}y^{bs-p+\b}(e(y))^{q}
\Psi(x,y;s-1)\psi(x)dy\right)dx,\\
W_{\a\b}^{(3)}(\psi)(s)&=\int_{E(\tau(r_1))}^{r_1}x^{as+\a}
\Psi(x,r_2;s) \psi(x)dx,\\
W_{\a\b}^{(4)}(\psi)(s)&=\int_{E(\tau(r_1))}^{r_1}
\left(\int_0^{r_2}x^{as-q+\a}y^{bs-p+\b}(e(y))^{q}
\Psi(x,y;s-1)
\psi(x) dy\right)dx.
\end{split}
\end{equation}
In order to prove that each $W_{\a\b}^{(j)}(\psi)(s)$ becomes
a holomorphic function on the half-plane ${\rm Re}(s)>-(\a+1)/a$,
it suffices to show the local uniform convergence of each $W_{\a\b}^{(j)}(\psi)(s)$
on the strip $-(\a+1)/a<{\rm Re}(s)\leq 0$.

Since $|\Psi(x,y;s)|\leq 1$ when ${\rm Re}(s)\leq 0$ and $(e(y))^q\leq x^{q}y^{q\delta}$ for $(x,y)\in U_1$,
we have
\begin{equation}
\left|W_{\a\b}^{(j)}(\psi)(s)\right|\leq \sup(|\psi|) \tilde{W}_{\a\b}^{(j)}(\sigma)
\quad \mbox{for $j=1,2,3,4$},
\end{equation}
where
\begin{equation}\label{eqn:5.23}
\begin{split}
\tilde{W}_{\a\b}^{(1)}(\sigma)&=\int_{0}^{E(\tau(r_1))}x^{a\sigma+\a}(E^{-1}(x))^{b\sigma+\b+1}dx,\\
\tilde{W}_{\a\b}^{(2)}(\sigma)&=\int_{0}^{E(\tau(r_1))}
\left(\int_0^{E^{-1}(x)}x^{a\sigma+\a}y^{b\sigma-p+\b+q\delta}
dy\right)dx,\\
\tilde{W}_{\a\b}^{(3)}(\sigma)&=\int_{E(\tau(r_1))}^{r_1}x^{a\sigma+\a}dx,\\
\tilde{W}_{\a\b}^{(4)}(\sigma)&=\int_{E(\tau(r_1))}^{r_1}
\left(\int_0^{r_2}x^{a\sigma+\a}y^{b\sigma-p+\b+q\delta}
dy\right)dx.
\end{split}
\end{equation}
Changing the integral variable in \eqref{eqn:5.23} by $u=E^{-1}(x)$,
we have
\begin{equation}\label{eqn:5.24}
\begin{split}
\tilde{W}_{\a\b}^{(1)}(\sigma)=&
-\delta\int_{0}^{\tau(r_1)}
u^{b\sigma+\b-a\delta\sigma-\a\delta-\delta}
(e(u))^{a\sigma+\a+1}du\\
&+\frac{p}{q}
\int_{0}^{\tau(r_1)}
u^{b\sigma-p+\b-a\delta\sigma-\a\delta-\delta}
(e(u))^{a\sigma+\a+1}du.
\end{split}
\end{equation}
Since $a\sigma+\a+1>0$, the flatness of $e(u)$ at the origin gives
the convergences of
the integrals in \eqref{eqn:5.24}.
From \eqref{eqn:5.8}, we see that 
$b\sigma-p+\b+q\delta>-1$,
which implies the convergences of the integrals $\tilde{W}_{\a\b}^{(2)}(\sigma)$ and $\tilde{W}_{\a\b}^{(4)}(\sigma)$.
The convergence of $\tilde{W}_{\a\b}^{(3)}(\sigma)$ is trivial.
The lemma is proven.
\end{proof}

\subsection{Holomorphic continuation of $Z_{\a\b}^{(2)}(\psi)(s)$ and $Z_{\a\b}^{(3)}(\psi)(s)$}

\begin{lemma}\label{lem:5.7}
The integrals
$Z_{\a\b}^{(2)}(\psi)(s)$ and $Z_{\a\b}^{(3)}(\psi)(s)$
become holomorphic functions
on the half-plane ${\rm Re}(s)>-(\a+1)/a$.
\end{lemma}

\begin{proof}
Let us show the local uniform convergences of the integrals
on the strip $-(\a+1)/a<{\rm Re}(s)\leq 0$.
The convergences on the half-plane ${\rm Re}(s)>0$ are trivial.

The integrals $Z_{\a\b}^{(2)}(\psi)(s)$ and $Z_{\a\b}^{(3)}(\psi)(s)$
can be estimated by using the following integrals:
\begin{equation}
\begin{split}
\tilde{Z}_{\a\b}^{(2)}(\sigma)&=\int_{U_2}x^{a\sigma+\a}y^{b\sigma+\b}dxdy,\\
\tilde{Z}_{\a\b}^{(3)}(\sigma)&=\int_{U_3}x^{(a-q)\sigma+\a}y^{b\sigma+\b}(e(y))^{q\sigma}dxdy.\\
\end{split}
\end{equation}
Indeed, we have that for $j=2,3$,
\begin{equation}
\begin{split}
\left|Z_{\a\b}^{(j)}(\psi)(s)\right|
&\leq \sup(|\psi|) \tilde{Z}_{\a\b}^{(j)}(\sigma).
\end{split}
\end{equation}

By decomposing the integral region $U_2$ 
into the following two sets:
\begin{equation}\label{eqn:5.27}
\{(x,y)\in U_2:y\leq \tau(r_1)\},\quad
\{(x,y)\in U_2:y> \tau(r_1)\},
\end{equation}
we have
\begin{equation}\label{eqn:5.28}
\tilde{Z}_{\a\b}^{(2)}(\sigma)=
F_1(\sigma)+F_2(\sigma)
\end{equation}
with
\begin{equation}
\begin{split}
F_1(\sigma)&=\int_{0}^{\tau(r_1)}\left(\int_{e(y)}^{E(y)} x^{a\sigma+\a}y^{b\sigma+\b}dx\right)dy,\\
F_2(\sigma)&=\int_{\tau(r_1)}^{\rho(r_1)}\left(\int_{e(y)}^{r_1} x^{a\sigma+\a}y^{b\sigma+\b}dx\right)dy.\\
\end{split}
\end{equation}
We remark that if $r_1\geq E(r_2)$, then the second set in \eqref{eqn:5.27} is empty.
The integral $F_1(\sigma)$ can be computed as
\begin{equation}\label{eqn:5.30}
F_1(\sigma)=
\frac{1}{(a\sigma+\a+1)}\int_0^{\tau(r_1)}
y^{b\sigma+\b}\left((E(y))^{a\sigma+\a+1}-(e(y))^{a\sigma+\a+1}\right)dy.
\end{equation}
The flatness of the functions $E(\cdot)$ and $e(\cdot)$ at the origin
gives the convergence of the integral in \eqref{eqn:5.30}.
The convergence of
the integral $F_2(\sigma)$ is clear.
From \eqref{eqn:5.28}, we obtain the convergence of $\tilde{Z}_{\a\b}^{(2)}(\sigma)$, which implies that of $Z_{\a\b}^{(2)}(\psi)(s)$.

Let us consider the integral $\tilde{Z}_{\a\b}^{(3)}(\sigma)$.
By decomposing the integral region $U_3$
into the following two sets:
\begin{equation}\label{eqn:5.31}
\{(x,y)\in U_3:y\leq \rho(r_1)\},\quad
\{(x,y)\in U_3:y> \rho(r_1)\},
\end{equation}
we have
\begin{equation}\label{eqn:5.32}
\tilde{Z}_{\a\b}^{(3)}(\sigma)=G_1(\sigma)+G_2(\sigma)
\end{equation}
with
\begin{equation}
\begin{split}
G_1(\sigma)&=\int_{0}^{\rho(r_1)}\left(\int_{0}^{e(y)}x^{(a-q)\sigma+\a}y^{b\sigma+\b}(e(y))^{q\sigma} dx\right)dy,\\
G_2(\sigma)&=\int_{\rho(r_1)}^{r_2}\left(\int_{0}^{r_1}x^{(a-q)\sigma+\a}y^{b\sigma+\b}(e(y))^{q\sigma} dx\right)dy.\\
\end{split}
\end{equation}
We remark that if $r_1\geq e(r_2)$, then the second set in \eqref{eqn:5.31} is empty.
The integral $G_1(\sigma)$ can be computed as
\begin{equation}\label{eqn:5.34}
G_1(\sigma)=\frac{1}{(a-q)\sigma+\a+1}
\int_{0}^{\rho(r_1)}y^{b\sigma+\b}(e(y))^{a\sigma+\a+1} dy.
\end{equation}
From the flatness of $e(\cdot)$ at the origin,
the integral $G_1(\sigma)$ converges.
On the other hand,
we have
\begin{equation}\label{eqn:5.35}
G_2(\sigma)=
\frac{r_1^{(a-q)\sigma+\a+1}}{(a-q)\sigma+\a+1}
\int_{\rho(r_1)}^{r_2}y^{b\sigma+\b}(e(y))^{q\sigma}dy.
\end{equation}
The last integral clearly converges.
Combining \eqref{eqn:5.32}, \eqref{eqn:5.34}, \eqref{eqn:5.35}
gives the convergence of $\tilde{Z}_{\a\b}^{(3)}(\sigma)$, which implies that of $Z_{\a\b}^{(3)}(\psi)(s)$.

Since $Z_{\a\b}^{(2)}(\psi)(s)$ and $Z_{\a\b}^{(3)}(\psi)(s)$ locally uniformly converge
on the half-plane ${\rm Re}(s)>-(\a+1)/a$,
the integrals become holomorphic functions there.
\end{proof}

\subsection{Proof of Theorem \ref{thm:5.2}}
From \eqref{eqn:5.10}, Lemmas \ref{lem:5.6} and \ref{lem:5.7},
the theorem is given.

\section{Asymptotic behavior of associated integrals at $s=-1/a$}
\label{sec:6}

In this section,
let us consider integrals of the form
\begin{equation}\label{eqn:6.1}
Z_{\b}(s)=\int_{V}\left|x^a y^b +x^{a-q} y^b e^{-1/|y|^p}\right|^s y^{\b}dxdy \quad
\mbox{for $s\in \C$},
\end{equation}
where 
$a,b,p,q$ are as in \eqref{eqn:3.1},
$\b$ is a nonnegative integer,
$V=\{(x,y)\in\R^2: 0\leq x \leq r_1, 0\leq y \leq r_2\}$
with $r_1, r_2\in (0,1)$
and $r_2$ is sufficiently small so that 
\begin{equation}\label{eqn:6.7.1}
r_2<\min\left\{\left(\frac{p}{p+1}\right)^{1/p},\left(\frac{p}{b/a+2p-1}\right)^{1/p}\right\}.
\end{equation}
From Remark \ref{rem:5.1} (ii),
the integral $Z_{\b}(s)$ becomes a holomorphic function
on the half-plane ${\rm Re}(s)>-1/a$
when $\b\geq b/a-1$.
On the other hand,
Theorem \ref{thm:5.2} and Remark \ref{rem:5.2} give that
$Z_{\b}(s)$ can be analytically continued as a meromorphic function, denoted again by $Z_{\b}(s)$, to the half-plane ${\rm Re}(s)>-1/a$,
which has only one simple pole at $s=-(\b+1)/b$,
when $\b<b/a-1$.
We show the following asymptotic behavior of
the restriction of $Z_{\b}(s)$
to the real axis as $s\to-1/a+0$.

\begin{theorem}\label{thm:6.1}
We have the following.
\begin{enumerate}
\item
If $0\leq \b< p+b/a-1$ and $\b\neq b/a-1$, then
\begin{equation}
\lim_{\sigma\to-1/a+0}(a\sigma +1)^{1+\frac{b/a-\b-1}{p}}\cdot Z_{\b}(\sigma)
=\frac{C_{\b}}{(-b/a+\b+1)},
\end{equation}
where $C_{\b}$ is the positive constant defined by
\begin{equation}\label{eqn:6.3}
C_{\b}=q^{\frac{b/a-\b-1}{p}}\Gamma\left(1+\frac{b/a-\b-1}{p}\right).
\end{equation}
\item
If $\b=b/a-1$,  
then
\begin{equation}
\lim_{\sigma\to-1/a+0}(a\sigma+1)^2\cdot Z_{\b}(\sigma)
=\frac{a}{b}.
\end{equation}
\item
If $\b =p+b/a-1$,
then
\begin{equation}
\lim_{\sigma\to-1/a+0}\left|\log(a\sigma +1)\right|^{-1}\cdot Z_{\b}(\sigma)
=\frac{1}{pq}.
\end{equation}.
\item
If $\b>p+b/a-1$, 
then there exists a constant $B_{\b}$ which
is independent of $\sigma$, such that
\begin{equation}
\lim_{\sigma\to -1/a+0}Z_{\b}(\sigma)=B_{\b}.
\end{equation}
\end{enumerate}
\end{theorem}

\subsection{Decompositions of associated integrals with a parameter}
In order to investigate the asymptotic behavior of $Z_{\b}(\sigma)$
as $\sigma \to -1/a+0$,
we decompose $Z_{\b}(s)$ into three integrals
with a parameter $\lambda$.

Let $\lambda\in (0,1)$.
From \eqref{eqn:6.7.1},
a positive real number $\delta$ is taken so that
\begin{equation}\label{eqn:6.7}
\max\left\{\frac{p+1}{q},\frac{b/a+2p-1}{q}\right\}<\delta<\frac{p}{qr_2^p}.
\end{equation}
We decompose the set $V$ into the following three sets:
\begin{equation}
\begin{split}
U_{1}(\lambda)&:=U_{1}^{(\delta)}(\lambda)=\{(x,y)\in V:\lambda^q x^q \geq y^{-q\delta}e^{-1/|y|^p}(\Leftrightarrow\lambda x \geq E(y))\},\\
U_{2}(\lambda)&:=U_{2}^{(\delta)}(\lambda)=\{(x,y)\in V:e^{-1/|y|^p}\leq \lambda^q x^q<y^{-q\delta}e^{-1/|y|^p} (\Leftrightarrow e(y) \leq \lambda x < E(y))\},\\
U_{3}(\lambda)&=\{(x,y)\in V:\lambda^q x^q <e^{-1/|y|^p}(\Leftrightarrow \lambda x < e(y))\},
\end{split}
\end{equation}
where $E(\cdot)$ is the smooth function defined in Section \ref{sec:5.1}
and $e(\cdot)$ is as in \eqref{eqn:4.1}.
We remark that
$U_1^{(\epsilon)}(\lambda)\subset U_1^{(\epsilon)}$ for $\epsilon>0$ and $U_3(\lambda) \supset U_3$, where $U_1^{(\cdot)}$ and $U_3$ are as in \eqref{eqn:5.9}.

Now, the integral $Z_{\b}(s)$ is expressed as
\begin{equation}\label{eqn:6.8}
Z_{\b}(s)=\sum_{j=1}^3 Z_{\b}^{(j)}(s;\lambda),
\end{equation}
where
\begin{equation}
Z_{\b}^{(j)}(s;\lambda)=\int_{U_j(\lambda)}\left|x^ay^b+x^{a-q}y^{b}e^{-1/|y|^p}\right|^s y^{\b}dxdy
\quad
\mbox{for $j=1,2,3$.}
\end{equation}
In a similar fashion to that of Section \ref{sec:5},
the integrals $Z_{\b}^{(j)}(s;\lambda)$ can be meromorphically continued to the half-plane
${\rm Re}(s)>-1/a$ as follows.

By decomposing the integral region $U_{1}(\lambda)$
into the following two sets:
\begin{equation}\label{eqn:6.10}
\left\{(x,y)\in U_1(\lambda):x\leq \frac{1}{\lambda}E(\tau(\lambda r_1))\right\}, \quad
\left\{(x,y)\in U_1(\lambda):x> \frac{1}{\lambda}E(\tau(\lambda r_1))\right\},
\end{equation}
we have
\begin{equation}
\begin{split}
Z_{\b}^{(1)}(s;\lambda)
=&\int_{0}^{\frac{1}{\lambda}E(\tau(\lambda r_1))}
\left(\int_0^{E^{-1}(\lambda x)}x^{as}y^{bs+\b}
\Psi(x,y;s)dy\right)dx\\
&+\int_{\frac{1}{\lambda}E(\tau(\lambda r_1))}^{r_1}
\left(\int_0^{r_2}x^{as}y^{bs+\b}
\Psi(x,y;s)dy\right)dx,
\end{split}
\end{equation}
where $\tau$ is as in \eqref{eqn:5.7} and $\Psi$ is as in \eqref{eqn:5.12}.
We remark that if $\lambda r_1< E(r_2)$, then the second set in \eqref{eqn:6.10} is empty.
From Lemma \ref{lem:5.4}, \eqref{eqn:5.15}, \eqref{eqn:6.7} and noticing $U_1^{(\delta)}(\lambda)\subset U_1^{(\delta)}$, 
integration by parts with respect to the variable $y$
gives
\begin{equation}\label{eqn:6.12}
\begin{split}
Z_{\b}^{(1)}(s;\lambda)
=&\frac{1}{bs+\b+1}W_{\b}^{(1)}(s;\lambda)-\frac{ps}{bs+\b+1}W_{\b}^{(2)}(s;\lambda)\\
&+\frac{r_2^{bs+\b+1}}{bs+\b+1}W_{\b}^{(3)}(s;\lambda)-\frac{ps}{bs+\b+1}W_{\b}^{(4)}(s;\lambda)
\end{split}
\end{equation}
with
\begin{equation}\label{eqn:6.13}
\begin{split}
W_{\b}^{(1)}(s;\lambda)&=\int_{0}^{\frac{1}{\lambda}E(\tau(\lambda r_1))}x^{as}(E^{-1}(\lambda x))^{bs+\b+1}\Psi(x,E^{-1}(\lambda x);s) dx,\\
W_{\b}^{(2)}(s;\lambda)&=\int_{0}^{\frac{1}{\lambda}E(\tau(\lambda r_1))}
\left(\int_0^{E^{-1}(\lambda x)}x^{as-q}y^{bs-p+\b}(e(y))^{q}
\Psi(x,y;s-1)dy\right)dx,\\
W_{\b}^{(3)}(s;\lambda)&=\int_{\frac{1}{\lambda}E(\tau(\lambda r_1))}^{r_1}x^{as}\Psi(x,r_2;s) dx,\\
W_{\b}^{(4)}(s;\lambda)&=\int_{\frac{1}{\lambda}E(\tau(\lambda r_1))}^{r_1}
\left(\int_0^{r_2}x^{as-q}y^{bs-p+\b}(e(y))^{q}\Psi(x,y;s-1)
dy\right)dx.
\end{split}
\end{equation}
The convergence of each integral $W_{\b}^{(j)}(s;\lambda)$
on the half-plane ${\rm Re}(s)>-1/a$ is shown 
in a similar fashion to the case of the integral $W_{\a\b}^{(j)}(\psi)(s)$, where $W_{\a\b}^{(j)}(\psi)(s)$ is as in \eqref{eqn:5.22}.
(See the proof of Lemma \ref{lem:5.6}.)
Hence, each $W_{\b}^{(j)}(s;\lambda)$
becomes a holomorphic function there.
From \eqref{eqn:6.12},
$Z_{\b}^{(1)}(s;\lambda)$ can be analytically continued as a meromorphic function, denoted again by $Z_{\b}^{(1)}(s;\lambda)$, to the half-plane ${\rm Re}(s)>-1/a$.
On the other hand,
we can show that $Z_{\b}^{(2)}(s;\lambda)$ and $Z_{\b}^{(3)}(s;\lambda)$ become holomorphic functions on the half-plane ${\rm Re}(s)>-1/a$
in a similar fashion to the proof of Lemma \ref{lem:5.7}.

To prove Theorem \ref{thm:6.1},
we investigate the asymptotic behavior of the restrictions of the integrals $Z_{\b}^{(j)}(s;\lambda)$ for $j=1,2,3$ to the real axis as $s\to -1/a+0$.

\subsection{Asymptotics of $Z_{\b}^{(1)}(s;\lambda)$}

\begin{lemma}\label{lem:6.2}
We have the following.
\begin{enumerate}
\item
If $0\leq \b< b/a-1$, then
\begin{equation}
\begin{split}
\frac{C_{\b}}{(-b/a+\b+1)}
&\leq\liminf_{\sigma\to -1/a+0}(a\sigma+1)^{1+\frac{b/a-\b-1}{p}}\cdot Z_{\b}^{(1)}(\sigma;\lambda)\\
&\leq
\limsup_{\sigma\to -1/a+0}(a\sigma+1)^{1+\frac{b/a-\b-1}{p}}\cdot Z_{\b}^{(1)}(\sigma;\lambda)
\leq
\frac{(1+\lambda^q)^{-1/a}C_{\b}}{(-b/a+\b+1)},
\end{split}
\end{equation}
where $C_{\b}$ is as in \eqref{eqn:6.3}.
\item
If $\b=b/a-1$,  
then
\begin{equation}
\begin{split}
\frac{(1+\lambda^q)^{-1/a}a}{b}
&\leq\liminf_{\sigma\to -1/a+0}(a\sigma+1)^{2}\cdot Z_{\b}^{(1)}(\sigma;\lambda)\\
&\leq
\limsup_{\sigma\to -1/a+0}(a\sigma+1)^{2}\cdot Z_{\b}^{(1)}(\sigma;\lambda)
\leq
\frac{a}{b}.
\end{split}
\end{equation}
\item
If $b/a-1 < \b< p+b/a-1$, then
\begin{equation}
\begin{split}
\frac{(1+\lambda^q)^{-1/a}C_{\b}}{(-b/a+\b+1)}
&\leq\liminf_{\sigma\to -1/a+0}(a\sigma+1)^{1+\frac{b/a-\b-1}{p}}\cdot Z_{\b}^{(1)}(\sigma;\lambda)\\
&\leq
\limsup_{\sigma\to -1/a+0}(a\sigma+1)^{1+\frac{b/a-\b-1}{p}}\cdot Z_{\b}^{(1)}(\sigma;\lambda)
\leq
\frac{C_{\b}}{(-b/a+\b+1)},
\end{split}
\end{equation}
where $C_{\b}$ is as in \eqref{eqn:6.3}.
\item
If $\b =p+b/a-1$,
then
\begin{equation}
\begin{split}
\frac{(1+\lambda^q)^{-1/a}}{pq}
&\leq\liminf_{\sigma\to -1/a+0}|\log(a\sigma +1)|^{-1}\cdot Z_{\b}^{(1)}(\sigma;\lambda)\\
&\leq
\limsup_{\sigma\to -1/a+0}|\log(a\sigma +1)|^{-1}\cdot Z_{\b}^{(1)}(\sigma;\lambda)
\leq
\frac{1}{pq}.
\end{split}
\end{equation}
\item
If $\b>p+b/a-1$, 
then there exists a constant $B_{\b}^{(1)}(\lambda)$ which
is independent of $\sigma$, such that
$\lim_{\sigma\to -1/a+0}Z_{\b}^{(1)}(\sigma;\lambda)=B_{\b}^{(1)}(\lambda)$.
\end{enumerate}
\end{lemma}

\begin{remark}
In the case (i) of Lemma \ref{lem:6.2}, $Z_{\b}^{(1)}(s;\lambda)$ is a meromorphic function on the half-plane ${\rm Re}(s)>-1/a$,
which has only one simple pole at $s=-(\b+1)/b$.
Hence, the limit of $Z_{\b}^{(1)}(\sigma;\lambda)$ as $\sigma\to -1/a+0$
takes a negative value in the case even though $Z_{\b}^{(1)}(s;\lambda)$ is originally defined by an integral whose integrand is nonnegative.
\end{remark}

First, we investigate the asymptotic behavior of the restriction of the integral $W_{\b}^{(1)}(s;\lambda)$ to the real axis as $s\to -1/a+0$, where $W_{\b}^{(1)}(s;\lambda)$ is as in \eqref{eqn:6.13}.

\begin{lemma}\label{lem:6.3}
We have the following.
\begin{enumerate}
\item
If $0\leq \b <p+b/a-1$,
then
\begin{equation}
\begin{split}
(1+\lambda^q)^{-1/a}C_{\b}
&\leq\liminf_{\sigma\to -1/a+0}(a\sigma+1)^{1+\frac{b/a-\b-1}{p}}\cdot W_{\b}^{(1)}(\sigma;\lambda)\\
&\leq
\limsup_{\sigma\to -1/a+0}(a\sigma+1)^{1+\frac{b/a-\b-1}{p}}\cdot W_{\b}^{(1)}(\sigma;\lambda)
\leq
C_{\b},
\end{split}
\end{equation}
where $C_{\b}$ is as in \eqref{eqn:6.3}.
\item
If $\b =p+b/a-1$,
then
\begin{equation}
\begin{split}
\frac{(1+\lambda^q)^{-1/a}}{q}
&\leq\liminf_{\sigma\to -1/a+0}|\log(a\sigma +1)|^{-1}\cdot W_{\b}^{(1)}(\sigma;\lambda)\\
&\leq
\limsup_{\sigma\to -1/a+0}|\log(a\sigma +1)|^{-1}\cdot W_{\b}^{(1)}(\sigma;\lambda)
\leq
\frac{1}{q}.
\end{split}
\end{equation}
\item
If $\b>p+b/a-1$, then there exists a constant $\tilde{B}_{\b}^{(1)}(\lambda)$ which
is independent of $\sigma$, such that
$\lim_{\sigma\to -1/a+0}W_{\b}^{(1)}(\sigma;\lambda)=\tilde{B}_{\b}^{(1)}(\lambda)$.
\end{enumerate}
\end{lemma}

\begin{proof}
For simplicity, we denote
\begin{equation}\label{eqn:6.20.1}
X:=a\sigma+1.
\end{equation}
It is clear that $X\to +0$ if and only if $\sigma \to -1/a+0$.

(i)
We notice that for $\sigma<0$ and $(x,y)\in U_1(\lambda)\cup U_2(\lambda)$,
\begin{equation}\label{eqn:6.18}
(1+\lambda^q)^{\sigma} \leq \Psi(x,y;\sigma)\leq 1.
\end{equation}
Thus, $W_{\b}^{(1)}(\sigma;\lambda)$
can be estimated by using the integral
\begin{equation}\label{eqn:6.19}
\tilde{W}_{\b}^{(1)}(\sigma;\lambda)=\int_{0}^{\frac{1}{\lambda}E(\tau(\lambda r_1))}x^{a\sigma}(E^{-1}(\lambda x))^{b\sigma+\b+1}dx.
\end{equation}
Indeed, 
the following inequalities hold from \eqref{eqn:6.18}:
\begin{equation}\label{eqn:6.20}
(1+\lambda^q)^{\sigma}\tilde{W}_{\b}^{(1)}(\sigma;\lambda)
\leq W_{\b}^{(1)}(\sigma;\lambda)
\leq \tilde{W}_{\b}^{(1)}(\sigma;\lambda).
\end{equation}
Changing the integral variable in \eqref{eqn:6.19} by $u=E^{-1}(\lambda x)$,
we have
\begin{equation}\label{eqn:6.21}
\begin{split}
\tilde{W}_{\b}^{(1)}(\sigma;\lambda)=&
-\frac{\delta}{\lambda^{X}}\int_{0}^{\tau(\lambda r_1)}u^{-b/a+\b+(b/a-\delta)X}
(e(u))^{X}du\\
&+\frac{p}{\lambda^{X}q}
\int_{0}^{\tau(\lambda r_1)}u^{-b/a-p+\b+(b/a-\delta)X}
(e(u))^{X}du.
\end{split}
\end{equation}
Lemma \ref{lem:4.1} and Remark \ref{rem:4.2} give that
if $0\leq \b <p+b/a-1$,
then
\begin{equation}\label{eqn:6.22}
\lim_{X\to+0}X^{1+\frac{b/a-\b-1}{p}}\cdot \tilde{W}_{\b}^{(1)}(\sigma;\lambda)
=q^{\frac{b/a-\b-1}{p}}\Gamma\left(1+\frac{b/a-\b-1}{p}\right).
\end{equation}
From \eqref{eqn:6.20} and \eqref{eqn:6.22},
we obtain the desired inequalities.

(ii)
Let $\b =p+b/a-1$.
Applying Lemma \ref{lem:4.1} to \eqref{eqn:6.21}, 
we have
\begin{equation}
\lim_{X\to+0}|\log X|^{-1}\cdot \tilde{W}_{\b}^{(1)}(\sigma;\lambda)
=\frac{1}{q}.
\end{equation}
In a similar fashion to the case (i),
we show the assertion.

(iii)
Assume $\b>p+b/a-1$.
From applying Lemma \ref{lem:4.1} to the two integrals in the right-hand side in \eqref{eqn:6.21},
the existence of the limit $\lim_{X\to+0}\tilde{W}_{\b}^{(1)}(\sigma;\lambda)$ is shown.
Then the boundedness of $\limsup_{X \to +0}W_{\b}^{(1)}(\sigma;\lambda)$
is obtained from \eqref{eqn:6.20}.
Applying Lemma \ref{lem:4.5} to $W_{\b}^{(1)}(\sigma;\lambda)$,
we have the convergence of $W_{\b}^{(1)}(\sigma;\lambda)$ as $X\to +0$.
\end{proof}

Next, we consider the integrals $W_{\b}^{(j)}(s;\lambda)$ for $j=2,3,4$, where $W_{\b}^{(j)}(s;\lambda)$ are as in \eqref{eqn:6.13}.
We remark that the restrictions of these integrals
to the real axis are nonnegative.

\begin{lemma}\label{lem:6.4}
For each $j=2,3,4$,
there exists a constant $\tilde{B}_{\b}^{(j)}(\lambda)$ which 
is independent of $\sigma$, such that
$\lim_{\sigma\to -1/a+0}W_{\b}^{(j)}(\sigma;\lambda)=\tilde{B}_{\b}^{(j)}(\lambda)$.
\end{lemma}

\begin{proof}
Since $\Psi(x,y;\sigma)\leq 1$ for $\sigma <0$ and $(e(y))^q\leq \lambda^{q} x^{q}y^{q\delta}$ for $(x,y)\in U_1(\lambda)$,
the integrals $W_{\b}^{(j)}(\sigma;\lambda)$
can be estimated by using the following integrals for $j=2,3,4$:
\begin{equation}\label{eqn:6.27}
\begin{split}
\tilde{W}_{\b}^{(2)}(\sigma;\lambda)&=\int_{0}^{\frac{1}{\lambda}E(\tau(\lambda r_1))}
\left(\int_0^{E^{-1}(\lambda x)}x^{a\sigma}y^{b\sigma-p+\b+q\delta}dy\right)dx,\\
\tilde{W}_{\b}^{(3)}(\sigma;\lambda)&=\int_{\frac{1}{\lambda}E(\tau(\lambda r_1))}^{r_1}x^{a\sigma}dx,\\
\tilde{W}_{\b}^{(4)}(\sigma;\lambda)&=\int_{\frac{1}{\lambda}E(\tau(\lambda r_1))}^{r_1}
\left(\int_0^{r_2}x^{a\sigma}y^{b\sigma-p+\b+q\delta}dy\right)dx.
\end{split}
\end{equation}
Indeed, we have
\begin{equation}\label{eqn:6.34}
\begin{split}
W_{\b}^{(2)}(\sigma;\lambda)
&\leq \lambda^{q}\tilde{W}_{\b}^{(2)}(\sigma;\lambda),\\
W_{\b}^{(3)}(\sigma;\lambda)
&\leq \tilde{W}_{\b}^{(3)}(\sigma;\lambda),\\
W_{\b}^{(4)}(\sigma;\lambda)
&\leq \lambda^{q}\tilde{W}_{\b}^{(4)}(\sigma;\lambda).
\end{split}
\end{equation}
Here, we easily have
\begin{equation}
\begin{split}
\tilde{W}_{\b}^{(2)}(\sigma;\lambda)
&=\int_0^{\tau(\lambda r_1)}\left(\int_{\frac{1}{\lambda}E(y)}^{\frac{1}{\lambda}E(\tau(\lambda r_1))}
x^{a\sigma}y^{b\sigma-p+\b+q\delta}dx\right)dy.\\
\end{split}
\end{equation}
Since the inequality:
$-b/a-2p+q\delta>-1$ holds from \eqref{eqn:6.7},
Lebesgue's convergence theorem gives 
the existences of the limits of the integrals $\tilde{W}_{\b}^{(2)}(\sigma;\lambda)$ and $\tilde{W}_{\b}^{(4)}(\sigma;\lambda)$ as $\sigma \to -1/a+0$.
Indeed,
\begin{equation}\label{eqn:6.36}
\begin{split}
\lim_{\sigma\to-1/a+0}\tilde{W}_{\b}^{(2)}(\sigma;\lambda)
&=\int_0^{\tau(\lambda r_1)}\left(\int_{\frac{1}{\lambda}E(y)}^{\frac{1}{\lambda}E(\tau(\lambda r_1))}
x^{-1}y^{-b/a-p+\b+q\delta}dx\right)dy\\
&=\int_0^{\tau(\lambda r_1)}
y^{-b/a-p+\b+q\delta}\left(\log\left(\frac{1}{\lambda}E(\tau(\lambda r_1))\right)-\log\left(\frac{1}{\lambda}E(y)\right)\right)dy\\
&=\int_0^{\tau(\lambda r_1)}
y^{-b/a-p+\b+q\delta}\left(\log\left(E(\tau(\lambda r_1))\right)+\delta\log y +\frac{y^{-p}}{q}\right)dy\\
\end{split}
\end{equation}
and
\begin{equation}\label{eqn:6.37}
\begin{split}
\lim_{\sigma\to -1/a+0}\tilde{W}_{\b}^{(4)}(\sigma;\lambda)
&=\int_{\frac{1}{\lambda}E(\tau(\lambda r_1))}^{r_1}
\left(\int_0^{r_2}x^{-1}y^{-b/a-p+\b+q\delta}dy\right)dx.\\
\end{split}
\end{equation}
We note that the integrals in \eqref{eqn:6.36} converge if $-b/a-2p+q\delta>-1$,
and the integral in \eqref{eqn:6.37} converges if $-b/a-p+q\delta>-1$.
(See the condition \eqref{eqn:6.7}.)
The existence of the limit of $\tilde{W}_{\b}^{(3)}(\sigma;\lambda)$ as $\sigma\to -1/a+0$
is trivial.
Indeed,
\begin{equation}\label{eqn:6.38}
\begin{split}
\lim_{\sigma\to -1/a+0}\tilde{W}_{\b}^{(3)}(\sigma;\lambda)
=\log r_1-\log \left(\frac{1}{\lambda}E(\tau(\lambda r_1))\right).
\end{split}
\end{equation}
From \eqref{eqn:6.34}, the above convergences imply
$\limsup_{\sigma\to -1/a+0}W_{\b}^{(j)}(\sigma;\lambda)<\infty$ for $j=2,3,4$.
Applying Lemma \ref{lem:4.5} to the integrals $W_{\b}^{(j)}(\sigma;\lambda)$,
we prove the lemma.
\end{proof}

\begin{proof}[Proof of Lemma \ref{lem:6.2}]
From \eqref{eqn:6.12}, Lemmas \ref{lem:6.3} and \ref{lem:6.4},
we obtain the lemma.
\end{proof}

\subsection{Asymptotics of $Z_{\b}^{(2)}(\sigma;\lambda)$}

\begin{lemma}\label{lem:6.5}
We have the following.
\begin{enumerate}
\item
If $
0 \leq \b < b/a-1$,
then
\begin{equation}
\begin{split}
(1+\lambda^q)^{-1/a}\delta A_{\b}
&\leq\liminf_{\sigma\to -1/a+0}(a\sigma+1)^{\frac{b/a-\b-1}{p}}|\log (a\sigma+1)|^{-1}\cdot Z_{\b}^{(2)}(\sigma;\lambda)\\
&\leq
\limsup_{\sigma\to -1/a+0}(a\sigma+1)^{\frac{b/a-\b-1}{p}}|\log (a\sigma+1)|^{-1}\cdot Z_{\b}^{(2)}(\sigma;\lambda)
\leq
\delta A_{\b},
\end{split}
\end{equation}
where
\begin{equation}\label{eqn:6.38.1}
A_{\b}=\frac{q^{(b/a-\b-1)/p}}{p^2}\Gamma\left(\frac{b/a-\b-1}{p}\right).
\end{equation}
\item
If $\b=b/a-1$, then
\begin{equation}
\begin{split}
\frac{(1+\lambda^q)^{-1/a}\delta}{2p^2}
&\leq\liminf_{\sigma\to -1/a+0}|\log (a\sigma+1)|^{-2}\cdot Z_{\b}^{(2)}(\sigma;\lambda)\\
&\leq
\limsup_{\sigma\to -1/a+0}|\log (a\sigma+1)|^{-2}\cdot Z_{\b}^{(2)}(\sigma;\lambda)
\leq
\frac{\delta}{2p^2}.
\end{split}
\end{equation}
\item
If $\b>b/a-1$,
then there exists a constant $B_{\b}^{(2)}(\lambda)$ which 
is independent of $\sigma$, such that
$\lim_{\sigma\to -1/a+0}Z_{\b}^{(2)}(\sigma;\lambda)=B_{\b}^{(2)}(\lambda)$.
\end{enumerate}
\end{lemma}

\begin{proof}
(i)
The integral $Z_{\b}^{(2)}(\sigma;\lambda)$
is represented as
\begin{equation}
Z_{\b}^{(2)}(\sigma;\lambda)
=\int_{U_2(\lambda)}x^{a\sigma}y^{b\sigma+\b}\Psi(x,y;\sigma)dxdy,
\end{equation}
where $\Psi$ is as in \eqref{eqn:5.12}.
From \eqref{eqn:6.18}, we have
\begin{equation}\label{eqn:6.40}
(1+\lambda^q)^{\sigma}\tilde{Z}_{\b}^{(2)}(\sigma;\lambda)
\leq Z_{\b}^{(2)}(\sigma;\lambda)
\leq \tilde{Z}_{\b}^{(2)}(\sigma;\lambda)
\end{equation}
with
\begin{equation}
\tilde{Z}_{\b}^{(2)}(\sigma;\lambda)=\int_{U_2(\lambda)}x^{a\sigma}y^{b\sigma+\b}dxdy.
\end{equation}
By decomposing the integral region $U_2(\lambda)$ into the following two sets:
\begin{equation}\label{eqn:6.41}
\{(x,y)\in U_2(\lambda):y\leq \tau(\lambda r_1)\},\quad
\{(x,y)\in U_2(\lambda):y> \tau(\lambda r_1)\},
\end{equation}
we have
\begin{equation}\label{eqn:6.42}
\tilde{Z}_{\b}^{(2)}(\sigma;\lambda)=
G^{\b}_1(\sigma)+G^{\b}_2(\sigma)
\end{equation}
with
\begin{equation}\label{eqn:6.43}
\begin{split}
G^{\b}_1(\sigma)&=\int_{0}^{\tau(\lambda r_1)}\left(\int_{\frac{1}{\lambda}e(y)}^{\frac{1}{\lambda}E(y)} x^{a\sigma}y^{b\sigma+\b}dx\right)dy,\\
G^{\b}_2(\sigma)&=\int_{\tau(\lambda r_1)}^{\rho(\lambda r_1)}\left(\int_{\frac{1}{\lambda}e(y)}^{r_1} x^{a\sigma}y^{b\sigma+\b}dx\right)dy,\\
\end{split}
\end{equation}
where $\tau$ and $\rho$ are as in \eqref{eqn:5.7}.
We remark that if $\lambda r_1 \geq E(r_2)$, then the second set in \eqref{eqn:6.41} is empty.

The integral $G^{\b}_1(\sigma)$ can be computed as
\begin{equation}\label{eqn:6.44}
\begin{split}
G^{\b}_1(\sigma)&=
\frac{1}{\lambda^{X}X}\int_0^{\tau(\lambda r_1)}
y^{-b/a+\b+(b/a-\delta)X}(e(y))^{X}dy\\
&\quad-\frac{1}{\lambda^{X}X}\int_0^{\tau(\lambda r_1)}y^{-b/a+\b+bX/a}(e(y))^{X}dy,
\end{split}
\end{equation}
where $X$ is as in \eqref{eqn:6.20.1}.
Applying Lemma \ref{lem:4.1} to the two integrals in \eqref{eqn:6.44}
and noticing Remark \ref{rem:4.2},
we have that
if $0\leq \b < b/a-1$, then
\begin{equation}\label{eqn:6.45}
\begin{split}
&\lim_{X\to +0}X^{\frac{b/a-\b-1}{p}}(\log X)^{-1}\cdot G^{\b}_1(\sigma)
=-\frac{\delta q^{(b/a-\b-1)/p}}{p^2}\Gamma\left(\frac{b/a-\b-1}{p}\right).
\end{split}
\end{equation}
On the other hand,
Lebesgue's convergence theorem gives that for any $\b\in \N$ 
\begin{equation}\label{eqn:6.47}
\begin{split}
\lim_{\sigma\to-1/a +0}G^{\b}_2(\sigma)
&=\int_{\tau(\lambda r_1)}^{\rho(\lambda r_1)}\left(\int_{\frac{1}{\lambda}e(y)}^{r_1} x^{-1}y^{-b/a+\b}dx\right)dy\\
&=\int_{\tau(\lambda r_1)}^{\rho(\lambda r_1)}
y^{-b/a+\b}\left(\log r_1+\log \lambda
+\frac{y^{-p}}{q}\right)dy.\\
\end{split}
\end{equation}
Notice that the last integral converges.
Combining \eqref{eqn:6.40}, \eqref{eqn:6.42}, \eqref{eqn:6.45} and \eqref{eqn:6.47} gives the assertion in the case (i).

(ii)
Applying Lemma \ref{lem:4.1} to the two integrals in \eqref{eqn:6.44},
we have that 
if $\b=b/a-1$, then
\begin{equation}
\lim_{X\to +0}|\log X|^{-2}\cdot G_1^{\b}(\sigma)
=\frac{\delta}{2p^2}.
\end{equation}
In a similar fashion to the proof of the case (i), we obtain the desired inequalities.

(iii)
From \eqref{eqn:6.43} and Lebesgue's convergence theorem,
we have
\begin{equation}\label{eqn:6.50}
\begin{split}
\lim_{X\to +0}G_1^{\b}(\sigma)
&=\int_{0}^{\tau(\lambda r_1)}\left(\int_{\frac{1}{\lambda}e(y)}^{\frac{1}{\lambda}E(y)} x^{-1}y^{-b/a+\b}dx\right)dy\\
&=-\delta\int_0^{\tau(\lambda r_1)}y^{-b/a+\b}(\log y) dy.\\
\end{split}
\end{equation}
We remark that the last improper integral converges
and can be computed by using integration by parts.
Combining \eqref{eqn:6.42}, \eqref{eqn:6.47} and  \eqref{eqn:6.50},
we show that
the limit $\lim_{\sigma\to -1/a+0}\tilde{Z}_{\b}^{(2)}(\sigma;\lambda)$
exists.
Then
the boundedness of $\limsup_{\sigma\to -1/a+0} Z_{\b}^{(2)}(\sigma;\lambda)$ is obtained from \eqref{eqn:6.40}.
Applying Lemma \ref{lem:4.5} to $Z_{\b}^{(2)}(\sigma;\lambda)$,
we show the existence of the limit of $Z_{\b}^{(2)}(\sigma;\lambda)$ as $\sigma \to -1/a+0$.
\end{proof}

\subsection{Asymptotics of $Z_{\b}^{(3)}(\sigma;\lambda)$}

\begin{lemma}\label{lem:6.6}
We have the following.
\begin{enumerate}
\item
If $0\leq \b < b/a-1$,
then
\begin{equation}
\begin{split}
\frac{\tilde{A}_{\b}}{(\lambda^{q}+1)^{1/a}}
&\leq \liminf_{\sigma\to -1/a+0}(a\sigma+1)^{\frac{b/a-\b-1}{p}}\cdot Z_{\b}^{(3)}(s;\lambda)\\
&\leq \limsup_{\sigma\to -1/a+0}(a\sigma+1)^{\frac{b/a-\b-1}{p}}\cdot Z_{\b}^{(3)}(s;\lambda)
\leq
\frac{\tilde{A}_{\b}}{\lambda^{q/a}},
\end{split}
\end{equation}
where 
\begin{equation}
\tilde{A}_{\b}=\frac{aq^{(b/a-\b-1)/p-1}}{p}\Gamma\left(\frac{b/a-\b-1}{p}\right).
\end{equation}
\item
If $\b=b/a-1$, then
\begin{equation}
\begin{split}
\frac{a}{(\lambda^{q}+1)^{1/a}pq}
&\leq\liminf_{\sigma\to -1/a+0}|\log (a\sigma+1)|^{-1}\cdot Z_{\b}^{(3)}(s;\lambda)\\
&\leq
\limsup_{\sigma\to -1/a+0}|\log (a\sigma+1)|^{-1}\cdot Z_{\b}^{(3)}(s;\lambda)
\leq
\frac{a}{\lambda^{q/a}pq}.
\end{split}
\end{equation}
\item
If $\b>b/a-1$,
then there exists a constant $B_{\b}^{(3)}(\lambda)$ which 
is independent of $\sigma$, such that
$\lim_{\sigma\to -1/a+0}Z_{\b}^{(3)}(\sigma;\lambda)=B_{\b}^{(3)}(\lambda)$.
\end{enumerate}
\end{lemma}

\begin{proof}
(i)
Since $\lambda x < e(y)$ on $U_3(\lambda)$,
we see that
\begin{equation}\label{eqn:6.56}
(1+\lambda^{-q})^{\sigma}\tilde{Z}_{\b}^{(3)}(\sigma;\lambda)
\leq Z_{\b}^{(3)}(\sigma;\lambda)\leq
\tilde{Z}_{\b}^{(3)}(\sigma;\lambda)
\end{equation}
with
\begin{equation}
\begin{split}
\tilde{Z}_{\b}^{(3)}(\sigma;\lambda)&=\int_{U_3(\lambda)}x^{(a-q)\sigma}y^{b\sigma+\b}(e(y))^{q\sigma}dxdy.
\end{split}
\end{equation}
By decomposing the integral region $U_3(\lambda)$ into the following two sets:
\begin{equation}\label{eqn:6.58}
\{(x,y)\in U_3(\lambda):y\leq \rho(\lambda r_1)\},\quad
\{(x,y)\in U_3(\lambda):y> \rho(\lambda r_1)\},
\end{equation}
we have
\begin{equation}\label{eqn:6.59}
\tilde{Z}_{\b}^{(3)}(\sigma;\lambda)=J_1(\sigma)+J_2(\sigma)
\end{equation}
with
\begin{equation}\label{eqn:6.60}
\begin{split}
J_1(\sigma)&=\int_{0}^{\rho(\lambda r_1)}\left(\int_{0}^{\frac{1}{\lambda}e(y)}x^{(a-q)\sigma}y^{b\sigma+\b}(e(y))^{q\sigma} dx\right)dy,\\
J_2(\sigma)&=\int_{\rho(\lambda r_1)}^{r_2}\left(\int_{0}^{r_1}x^{(a-q)\sigma}y^{b\sigma+\b}(e(y))^{q\sigma} dx\right)dy,\\
\end{split}
\end{equation}
where $\rho$ is as in \eqref{eqn:5.7}.
We remark that if $\lambda r_1 \geq e(r_2)$, then the second set in \eqref{eqn:6.58} is empty.
The integral $J_1(\sigma)$ can be computed as
\begin{equation}\label{eqn:6.61}
J_1(\sigma)=\frac{1}{\lambda^{X-qX/a+q/a}(X-qX/a+q/a)}
\int_{0}^{\rho(\lambda r_1)}y^{-b/a+\b+bX/a}(e(y))^{X} dy,
\end{equation}
where $X$ is as in \eqref{eqn:6.20.1}.
Lemma \ref{lem:4.1} and Remark \ref{rem:4.2} give that
if $0\leq \b<b/a-1$, then
\begin{equation}\label{eqn:6.62}
\lim_{X\to +0}X^{\frac{b/a-\b-1}{p}}\cdot J_1(\sigma)=
\frac{aq^{(b/a-\b-1)/p-1}}{\lambda^{q/a}p}\Gamma\left(\frac{b/a-\b-1}{p}\right).
\end{equation}
On the other hand,
Lebesgue's convergence theorem gives that
for any $\b\in \N$,
\begin{equation}\label{eqn:6.63}
\begin{split}
\lim_{\sigma\to -1/a+0}J_2(\sigma)
&=\int_{\rho(\lambda r_1)}^{r_2}\left(\int_{0}^{r_1}x^{q/a-1}y^{-b/a+\b}(e(y))^{-q/a} dx\right)dy\\
&=\frac{a r_1^{q/a}}{q}\int_{\rho(\lambda r_1)}^{r_2}y^{-b/a+\b}(e(y))^{-q/a}dy.
\end{split}
\end{equation}
We remark that the last integral converges.
Combining \eqref{eqn:6.56}, \eqref{eqn:6.59}, \eqref{eqn:6.62} and \eqref{eqn:6.63} gives the assertion.

(ii)
Applying Lemma \ref{lem:4.1} to \eqref{eqn:6.61},
we have that
if $\b=b/a-1$, then
\begin{equation}
\lim_{X\to +0}|\log X|^{-1}\cdot J_1(\sigma)=\frac{a}{\lambda^{q/a}pq}.
\end{equation}
We similarly show the desired inequalities to the case (i).

(iii)
Analogous to the proof of Lemma \ref{lem:6.5} (iii),
it suffices to show the existence of the limit of the integral $\tilde{Z}_{\b}^{(3)}(\sigma;\lambda)$
as $\sigma \to -1/a+0$.
From applying Lemma \ref{lem:4.1} to \eqref{eqn:6.61},
it is shown that
if $\b>b/a-1$, then the limit $\lim_{\sigma\to-1/a+0}J_1(\sigma)$ exists. 
Hence
the limit $\lim_{\sigma\to -1/a+0}\tilde{Z}_{\b}^{(3)}(\sigma;\lambda)$
also exists
from \eqref{eqn:6.59} and \eqref{eqn:6.63}.
\end{proof}

\subsection{Proof of Theorem \ref{thm:6.1}}
Let us consider the case where $0\leq \b<b/a-1$.
From \eqref{eqn:6.8}, Lemmas \ref{lem:6.2}, \ref{lem:6.5} and \ref{lem:6.6},
\begin{equation}\label{eqn:6.70}
\begin{split}
\frac{C_{\b}}{(-b/a+\b+1)}
&\leq
\liminf_{\sigma\to -1/a+0}(a\sigma+1)^{1+\frac{b/a-\b-1}{p}}\cdot Z_{\b}^{(1)}(\sigma;\lambda)\\
&=
\liminf_{\sigma\to-1/a+0}(a\sigma +1)^{1+\frac{b/a-\b-1}{p}}\cdot Z_{\b}(\sigma)\\
&\leq 
\limsup_{\sigma\to-1/a+0}(a\sigma +1)^{1+\frac{b/a-\b-1}{p}}\cdot Z_{\b}(\sigma)\\
&=
\limsup_{\sigma\to -1/a+0}(a\sigma+1)^{1+\frac{b/a-\b-1}{p}}\cdot Z_{\b}^{(1)}(\sigma;\lambda)\\
&\leq
\frac{(1+\lambda^q)^{-1/a}C_{\b}}{(-b/a+\b+1)},
\end{split}
\end{equation} 
where
$C_{\b}$ is as in \eqref{eqn:6.3}.
Note that $Z_{\b}(\sigma)$ is independent of $\lambda$.
Considering the limit as $\lambda\to 0$ in \eqref{eqn:6.70},
we have
\begin{equation}\label{eqn:6.71}
\lim_{\sigma \to -1/a+0}(a\sigma+1)^{1+\frac{b/a-\b-1}{p}}\cdot Z_{\b}(\sigma)=\frac{C_{\b}}{(-b/a+\b+1)}.
\end{equation}
In the case where $b/a-1<\b<p+b/a-1$, we similarly obtain \eqref{eqn:6.71}, which completes the proof in the case (i).
Analogously, we can show the assertions (ii), (iii) in Theorem \ref{thm:6.1}.
Notice that
$C_{\b}=1$ if $\b=b/a-1$.

Assume $\b>p+b/a-1$.
From \eqref{eqn:6.8}, Lemmas \ref{lem:6.2}, \ref{lem:6.5} and \ref{lem:6.6},
the limit of $Z_{\b}(\sigma)$
as $\sigma \to -1/a+0$ exists.
Since $Z_{\b}(\sigma)$ is independent of $\lambda$,
so is the limit of it.
This completes the proof of the theorem.

\subsection{Asymptotic limits of slightly general integrals}

In the case where $\b>p+b/a-1$,
we consider the integrals 
\begin{equation}\label{eqn:6.72}
\begin{split}
\hat{Z}_{\b}(\phi)(s)
&=\int_{V}\left|x^a y^b +x^{a-q} y^b e^{-1/|y|^p}\right|^s y^{\b}\phi(x,y)dxdy \quad
\mbox{for $s\in \C$},
\end{split}
\end{equation}
where $a,b,p,q,\b,V$ are as in \eqref{eqn:6.1}
and $\phi$ is a smooth function defined on $V$.
Letting $\a=0$ in \eqref{eqn:5.4}
and noticing Remark \ref{rem:5.1} (i),
we see that
if $\b\geq b/a-1$, then
the integral $\hat{Z}_{\b}(\phi)(s)$
can be analytically continued as a holomorphic
function
to the half-plane ${\rm Re}(s)>-1/a$.
Let us investigate the behavior at $s=-1/a$ of the restriction of $\hat{Z}_{\b}(\phi)(s)$ to the real axis.

\begin{lemma}\label{lem:6.7}
If $\b>p+b/a-1$, 
then there exists a constant $\hat{B}_{\b}(\phi)$ which is independent of $\sigma$, such that
$\lim_{\sigma\to -1/a+0}\hat{Z}_{\b}(\phi)(\sigma)=\hat{B}_{\b}(\phi)$.
\end{lemma}

\begin{proof}
Let $X$ be as in \eqref{eqn:6.20.1} for simplicity.
Theorem \ref{thm:6.1} (iv) implies
the boundedness of 
$\limsup_{X\to +0}Z_{\b}(\sigma)$
for $\b\in \Z_{+}$ satisfying $\b>p+b/a-1$,
where $Z_{\b}(s)$ is as in \eqref{eqn:6.1}.
From Applying Lemma \ref{lem:4.6}
to the integral $\hat{Z}_{\b}(\phi)(\sigma)$,
the limit of 
$\hat{Z}_{\b}(\phi)(\sigma)$ as $X\to +0$ exists.
\end{proof}

\section{Proof of Theorem \ref{thm:3.1}}
Let $f$ be as in \eqref{eqn:3.1}, 
$\varphi$ as in \eqref{eqn:1.1}
and
\begin{equation}
\tilde{V}=\{(x,y)\in \R^2:-r_1\leq x \leq r_1, -r_2\leq y \leq r_2\},
\end{equation}
where $r_1,r_2 \in (0,1)$ and $r_2$ is sufficiently small so that
\begin{equation}\label{eqn:7.1}
r_2<\min\left\{
\left(\frac{p}{2b/a+p-1}\right)^{1/p},
\left(\frac{p}{b/a+2p-1}\right)^{1/p}\right\}.
\end{equation}
Note that
$r_2$ satisfies the conditions \eqref{eqn:5.2.0} for $\a=0,1$
and \eqref{eqn:6.7.1}
since $a\leq b$.
We decompose 
\begin{equation}\label{eqn:7.2}
Z_f(\varphi)(s)=I(\varphi)(s)+J(\varphi)(s)
\end{equation}
with
\begin{equation}
\begin{split}
I(\varphi)(s)&=\int_{\tilde{V}}|f(x,y)|^s\varphi(x,y)\chi(x,y)dxdy,\\
J(\varphi)(s)&=\int_{\R^2}|f(x,y)|^s\varphi(x,y)(1-\chi(x,y))dxdy,\\
\end{split}
\end{equation}
where $\chi:\R^2\to[0,1]$ is a cut-off function
satisfying that $\chi=1$ near the origin, 
the support of $\chi$ is contained in $\tilde{V}$,
and $\chi(|x|,|y|)=\chi(x,y)$.

First we consider the integral $I(\varphi)(s)$,
which is essentially important.

\begin{lemma}\label{lem:7.1}
The integral $I(\varphi)(s)$ can be analytically continued as a meromorphic function,
denoted again by $I(\varphi)(s)$, to the half-plane ${\rm Re}(s)>-1/a$,
and its poles are contained in the set $\{-j/b:j\in \N \mbox{ with } j<b/a\}$.
Moreover, the following hold. 
\begin{enumerate}
\item
If at least one of the following two conditions is satisfied:
\begin{enumerate} 
\item
$0<p<b/a-1$,
\item
$b/a$ is not an odd integer,
\end{enumerate}
then
\begin{equation}
\lim_{\sigma\to -1/a+0}(a\sigma+1)^{1+\frac{b/a-1}{p}}
\cdot I(\varphi)(\sigma)
=-4C\cdot \varphi(0,0),
\end{equation}
where $C$ is as in \eqref{eqn:3.3}.
\item
If $p=b/a-1$ and $b/a$ is an odd integer,
then
\begin{equation}
\lim_{\sigma\to -1/a+0}(a\sigma+1)^{2}
\cdot I(\varphi)(\sigma) 
=-\frac{4q}{p}\cdot \varphi(0,0)
+\frac{4a}{b p!}\cdot \frac{\d^{p}\varphi}{\d y^{p}}(0,0).
\end{equation}
\item
If $p>b/a-1$ and $b/a$ is an odd integer,
then
\begin{equation}
\lim_{\sigma\to -1/a+0}(a\sigma+1)^{2}
\cdot I(\varphi)(\sigma) 
=\frac{4a}{b m!}\cdot \frac{\d^{m}\varphi}{\d y^{m}}(0,0),
\end{equation}
where $m=b/a-1$.
\end{enumerate}
\end{lemma}

\begin{proof}
Since $q$ is even,
$|f(\theta_1 x,\theta_2 y)|=|f(x,y)|$ for any $(x,y)\in \R^2$ and $\theta_1,\theta_2\in\{1,-1\}$.
Using the orthant decomposition, we have
\begin{equation}\label{eqn:7.7}
I(\varphi)(s)=\sum_{\theta \in\{1,-1\}^2}
I_{+}(\varphi_{\theta})(s)
\end{equation}
with
\begin{equation}\label{eqn:7.8}
I_{+}(\varphi)(s)=\int_{V}|f(x,y)|^s\varphi(x,y)\chi(x,y)dxdy,
\end{equation}
where
$\varphi_{\theta}(x,y)=\varphi(\theta_1 x,\theta_2 y)$
with $\theta=(\theta_1,\theta_2)$ and
$V=\{(x,y)\in\R^2: 0\leq x \leq r_1, 0\leq y \leq r_2\}$.
Let $M=\max\{\b\in\Z :0\leq \b\leq p+b/a-1\}$.
Using Taylor's formula,
we have
\begin{equation}\label{eqn:7.9}
\varphi_{\theta}(x,y)\chi(x,y)=\sum_{\b=0}^{M}\frac{\theta_2^{\b}}{\b !}\frac{\d^{\b}\varphi}{\d y^{\b}}(0,0)y^{\b}
+\sum_{\b=0}^{M}xy^{\b}\varphi_{\theta}^{(\b)}(x)
+y^{M+1}\tilde{\varphi}_{\theta}(x,y),
\end{equation}
where
\begin{equation}
\begin{split}
\varphi_{\theta}^{(\b)}(x)&=\frac{1}{\b!}\int_{0}^{1}\frac{\d^{\b+1}(\varphi_{\theta} \chi)}{\d x\d y^{\b}}(tx,0)dt,\\
\tilde{\varphi}_{\theta}(x,y)&=\frac{1}{M!}\int_0^1(1-t)^{M}\frac{\d^{M+1}(\varphi_{\theta} \chi)}{\d y^{M+1}}(x,ty)dt,
\end{split}
\end{equation}
with $(\varphi_{\theta} \chi)(x,y)=\varphi_{\theta}(x,y)\chi(x,y)$.
Here, $\varphi_{\theta}^{(\b)}$ and $\tilde{\varphi}_{\theta}$ are smooth functions on $[-r_1,r_1]$ and $\tilde{V}$ respectively.
Substituting \eqref{eqn:7.9} to \eqref{eqn:7.7},
we have
\begin{equation}\label{eqn:7.11}
I(\varphi)(s)=
\sum_{k=0}^{\tilde{M}}\frac{4}{(2k)!}\frac{\d^{2k}\varphi}{\d y^{2k}}(0,0)Z_{(2k)}(s)
+R(s)
\end{equation}
with $\tilde{M}=\max\{k\in \Z:2k\leq M\}$ and
\begin{equation}\label{eqn:7.12}
R(s)=\sum_{\theta\in\{1,-1\}^2}\left(\sum_{\b=0}^{M}Z_{1\b}(\varphi_{\theta}^{(\b)})(s)
+\hat{Z}_{M+1}(\tilde{\varphi}_{\theta})(s)\right),
\end{equation}
where 
$Z_{\b}(s)$, $Z_{\a\b}(\psi)(s)$, $\hat{Z}_{\b}(\phi)(s)$
are as in \eqref{eqn:6.1}, \eqref{eqn:5.1}, \eqref{eqn:6.72} respectively.

The integral $Z_{\b}(s)$ is precisely investigated in Theorem \ref{thm:6.1}.
On the other hand,
we have the following.

\begin{lemma}\label{lem:7.2}
Let $R(s)$ be as in \eqref{eqn:7.12}.
Then $R(s)$ can be analytically continued as
a meromorphic function, 
denoted again by $R(s)$, 
to the half-plane ${\rm Re}(s)>-1/a$,
and its poles are contained in the set 
$\{-j/b:j\in\N \mbox{ with } j<b/a\}$.
In particular, the following hold.
\begin{enumerate}
\item
If $b/a-1$ is not an integer,
then
the limit of $R(\sigma)$
as $\sigma\to -1/a+0$ exists.  
\item
If $b/a-1$ is an integer, then
the limit of $(a\sigma+1)\cdot R(\sigma)$
as $\sigma\to -1/a+0$ exists. 
\end{enumerate}
\end{lemma}
\begin{proof}[Proof of Lemma \ref{lem:7.2}]
From Theorem \ref{thm:5.2},
the integral $Z_{1\b}(\varphi_{\theta}^{(\b)})(s)$
can be analytically continued as a meromorphic function to the half-plane ${\rm Re}(s)>-2/a$, which has only one pole at $s=-(\b+1)/b$ of order at most one.
In particular, we have that if $\b\neq b/a-1$, then the limit of $Z_{1\b}(\varphi_{\theta}^{(\b)})(\sigma)$ as $\sigma\to -1/a+0$ exists, and if $\b= b/a-1$, then the limit of $(a\sigma+1)\cdot Z_{1\b}(\varphi_{\theta}^{(\b)})(\sigma)$ as $\sigma\to -1/a+0$ exists.
Remark \ref{rem:5.1} (i) and Lemma \ref{lem:6.7} give that
the integral $\hat{Z}_{M+1}(\tilde{\varphi}_{\theta})(s)$
becomes a holomorphic function
on the half-plane ${\rm Re}(s)>-1/a$
and the limit of $\hat{Z}_{M+1}(\tilde{\varphi}_{\theta})(\sigma)$
as $\sigma\to -1/a+0$ exists. 
Hence, we show the lemma from \eqref{eqn:7.12}.
\end{proof}

From \eqref{eqn:7.11}, Theorem \ref{thm:5.2} and Lemma \ref{lem:7.2},
the integral $I(\varphi)(s)$ can be analytically continued as a meromorphic function, denoted again by $I(\varphi)(s)$,
to the half-plane ${\rm Re}(s)>-1/a$,
and its poles are contained in the set $\{-j/b:j\in \N \mbox{ with } j<b/a\}$.

(i)
First,
we assume that $b/a$ is not an odd integer.
We remark that if $b/a$ is an even integer,
then the integral $Z_m(s)$ does not appear in \eqref{eqn:7.11},
where $m:=b/a-1$.
Hence, Theorem \ref{thm:6.1} gives that if $b/a$ is not an odd integer,
then the growth rate of the integrals $Z_{(2k)}(\sigma)$ in \eqref{eqn:7.11}
as $\sigma\to -1/a+0$ decreases with respect to $k$.
Therefore,
we have from \eqref{eqn:7.11}, Theorem \ref{thm:6.1} and Lemma \ref{lem:7.2}
that
\begin{equation}\label{eqn:7.13}
\begin{split}
\lim_{\sigma\to -1/a+0}(a\sigma+1)^{1+\frac{b/a-1}{p}}
\cdot I(\varphi)(\sigma)
&=\lim_{\sigma\to -1/a+0}4\varphi(0,0)\cdot (a\sigma+1)^{1+\frac{b/a-1}{p}} \cdot Z_0(\sigma)\\
&=-\frac{4q^{(b/a-1)/p}}{(b/a-1)}\Gamma\left(1+\frac{b/a-1}{p}\right)\cdot \varphi(0,0).
\end{split}
\end{equation}

Next, we assume that $0<p<b/a-1$ and $b/a$ is an odd integer.
We denote $m:=b/a-1$.
Since $m$ is an even integer,
the integral $Z_m(s)$ appears in \eqref{eqn:7.11}.
Here, Theorem \ref{thm:6.1} gives
\begin{equation}
\lim_{\sigma\to-1/a+0}(a\sigma+1)^2\cdot Z_{m}(\sigma)
=\frac{a}{b}.
\end{equation}
From \eqref{eqn:7.11}, Theorem \ref{thm:6.1}, Lemma \ref{lem:7.2} and comparing
the growth rate of $Z_0(\sigma)$ and $Z_m(\sigma)$ as $\sigma\to -1/a+0$,
we also have \eqref{eqn:7.13}.

(ii)
Let us consider the case where $p=b/a-1$ and $b/a$ is an odd integer.
In a similar fashion to the case where $0<p<b/a-1$ and $b/a$ is an odd integer,
we have
\begin{equation}
\begin{split}
&\lim_{\sigma\to -1/a+0}(a\sigma+1)^{2}
\cdot I(\varphi)(\sigma)\\
=&\lim_{\sigma\to -1/a+0}\left(4\varphi(0,0)\cdot(a\sigma+1)^{2}
\cdot Z_0(\sigma)+\frac{4}{p!}\frac{\d^{p}\varphi}{\d y^{p}}(0,0)\cdot(a\sigma+1)^{2}
\cdot Z_p(\sigma) \right)\\
=&-\frac{4q}{p}\cdot \varphi(0,0)
+\frac{4a}{b p!}\cdot \frac{\d^{p}\varphi}{\d y^{p}}(0,0).
\end{split}
\end{equation}

(iii)
Assume that $p>b/a-1$ and $b/a$ is an odd integer.
We obtain the following equations analogously to the above cases
where $b/a$ is an odd integer.
\begin{equation}
\begin{split}
\lim_{\sigma\to -1/a+0}(a\sigma+1)^{2}
\cdot I(\varphi)(\sigma)
&=
\lim_{\sigma\to -1/a+0}\frac{4}{m!}\frac{\d^{m}\varphi}{\d y^{m}}(0,0)\cdot(a\sigma+1)^{2}
\cdot Z_m(\sigma)\\
&=\frac{4a}{b m!}\cdot \frac{\d^{m}\varphi}{\d y^{m}}(0,0),
\end{split}
\end{equation}
where $m=b/a-1$.
\end{proof}

Next, we consider the integral $J(\varphi)(s)$. 

\begin{lemma}\label{lem:7.3}
The integral $J(\varphi)(s)$ can be analytically continued as a meromorphic function, denoted again by $J(\varphi)(s)$, to the whole complex plane and its poles are contained in
the set $\{-j/(a-q),-k/b:j,k\in \N\}$.
In particular, the following hold.
\begin{enumerate}
\item
If $b/a$ is not an integer,
then 
the limit of
$J(\varphi)(\sigma)$
as $\sigma\to -1/a+0$
exists.
\item
If $b/a$ is an integer,
then
the limit of
$(a\sigma+1)\cdot J(\varphi)(\sigma)$
as $\sigma\to -1/a+0$
exists.
\end{enumerate}
\end{lemma}

\begin{proof}
In a similar fashion to the case of $I(\varphi)(s)$,
the orthant decomposition gives
\begin{equation}\label{eqn:7.15}
J(\varphi)(s)=\sum_{\theta\in \{1,-1\}^2}J_{+}(\varphi_{\theta})(s),
\end{equation}
where
\begin{equation}
J_{+}(\varphi)(s)=\int_{\R_{+}^2}|f(x,y)|^s\varphi(x,y)(1-\chi(x,y))dxdy
\end{equation}
and $\varphi_{\theta}$ is as in \eqref{eqn:7.7}.
Here, let us investigate $J_{+}(\varphi)(s)$.
We have
\begin{equation}
\begin{split}
J_{+}(\varphi)(s)&=\int_{\R_+^2}x^{(a-q)s} y^{bs}\Phi(x,y;s)dxdy,
\end{split}
\end{equation}
where $\Phi(x,y;s)=|x^{q}+e^{-1/|y|^{p}}|^{s}\varphi(x,y)(1-\chi(x,y))$.
Since $x^{q}+e^{-1/|y|^{p}}$ does not vanish on $\R^2\setminus \{(0,0)\}$,
we see that 
$\Phi(\cdot;s)$ is a smooth function on $\R^2$ whose support is compact
for all $s\in\C$
and
$\frac{\d^{\a+\b}\Phi}{\d x^{\a}\d y^{\b}}(x,y;\cdot)$
is an entire function on $\C$ for all $(x,y)\in \R^2$, $\a,\b\in\Z_{\geq 0}$.

The following lemma is essentially known in \cite{gs64}, \cite{bg69}, \cite{agv88}, \cite{kn20}.

\begin{lemma}[Lemma 4.1 and Remark 4.2 of \cite{kn20}]\label{lem:7.4}
Let $A$, $B$ be integers with $A>0$, $B\geq 0$ 
and let $\eta(u;s)$ be a complex-valued function defined on $\R\times \C$.
We assume that
\begin{enumerate}
\item[(a)]
$\eta(\cdot;s)$ is a smooth function on $\R$ whose support is compact
for all $s\in\C$.
\item[(b)]
$\frac{\d^{\a}\eta}{\d u^{\a}}(u;\cdot)$
is an entire function on $\C$ for all $u\in \R$, $\a\in\Z_{+}$.
\end{enumerate} 
Let
\begin{equation}
H(s):=\int_{0}^{\infty}x^{As+B}\eta(u;s)du.
\end{equation}
Then the following hold.
\begin{enumerate}
\item
The integral $H(s)$ becomes a holomorphic function on the half-plane ${\rm Re}(s)>-(B+1)/A$.
\item
The integral $H(s)$ can be analytically continued as a meromorphic function to the whole complex plane.
Moreover, its poles are simple and they are contained in the set
$\{-(B+j)/A:j\in\N\}$.
\end{enumerate}
\end{lemma}

From Lemma \ref{lem:7.4},
the integral $J_{+}(\varphi)(s)$ can be analytically continued as a meromorphic function, denoted again by $J_{+}(\varphi)(s)$, to the whole complex plane and its poles are contained in
the set $\{-j/(a-q),-k/b:j,k\in \N\}$.
Moreover, its poles at $s=-k/b$ contained in the half-plane ${\rm Re}(s)>-1/(a-q)$ are of order at most one.
In particular, 
if $b/a$ is not an integer,
then 
the limit of
$J_{+}(\varphi)(\sigma)$
as $\sigma\to -1/a+0$
exists, and
if $b/a$ is an integer,
then
the limit of
$(a\sigma+1)\cdot J_{+}(\varphi)(\sigma)$
as $\sigma\to -1/a+0$
exists.
From \eqref{eqn:7.15},
the integral $J(\varphi)(s)$ has the same properties to that of $J_{+}(\varphi)(s)$.
\end{proof}

Finally,
Theorem \ref{thm:3.1} can be shown by applying Lemmas \ref{lem:7.1} and \ref{lem:7.3} to \eqref{eqn:7.2}.

\section{Proof of Theorem \ref{thm:3.3}}

Let us consider the integral
\begin{equation}\label{eqn:8.1}
Z(\sigma)=\int_{V}
\left|x^ay^b+x^a y^{b-q}e^{-1/|x|^p}+x^{a-\tilde{q}}y^{b}e^{-1/|y|^{\tilde{p}}}\right|^{\sigma}dxdy \quad \mbox{for $\sigma<0$},
\end{equation}
where $a,b,p,q,\tilde{p}, \tilde{q}$
are as in \eqref{eqn:3.7} and
$V=\{(x,y)\in\R^2:0\leq x\leq r_1,0\leq y\leq r_2\}$
with $r_1,r_2\in (0,1)$.
A similar estimate to \eqref{eqn:5.3} for $Z(\sigma)$
implies that $Z(\sigma)$ converges if $\sigma>-1/b$.

\subsection{Asymptotics of $Z(\sigma)$}

\begin{theorem}\label{thm:8.1}
We have the following.
\begin{enumerate}
\item
If $p>1-a/b$, 
then
\begin{equation}\label{eqn:8.2}
\lim_{\sigma\to-1/b+0}
(b\sigma+1)^{1-\frac{1-a/b}{p}}\cdot Z(\sigma)
=\hat{C},
\end{equation}
where $\hat{C}$ is as in \eqref{eqn:3.10}.
\item
If $p=1-a/b$, 
then 
\begin{equation}
\lim_{\sigma\to-1/b+0}
|\log (b\sigma+1)|^{-1}\cdot Z(\sigma)
=\frac{1}{pq}.
\end{equation}
\item
If $0<p< 1-a/b$, then the limit of 
$Z(\sigma)$ as 
$\sigma\to-1/b+0$ exists.
\end{enumerate}
\end{theorem}

In order to prove Theorem \ref{thm:8.1},
we investigate the asymptotic limit of the integral $W(\sigma)$ in \eqref{eqn:8.5} below.
Let $\lambda, \mu$ be small positive real numbers.
We denote $e(x):=e_{p,q}(x)$ and $\tilde{e}(y):=e_{\tilde{p},\tilde{q}}(y)$, where $e_{p,q}$ is as in \eqref{eqn:4.1}.
We remark that $e$, $\tilde{e}$ are monotonously increasing on $[0,\infty)$.
Since $\frac{de}{dx}(0)=0$ and $\frac{d\tilde{e}}{dy}(0)$=0,
there exists $r\in (0,r_1]$ such that
$\frac{1}{\lambda}e(x)<\tilde{e}^{-1}(\mu x)<r_2$ on $(0,r)$.
Here, we denote
\begin{equation}\label{eqn:8.4.0}
\begin{split}
U(\lambda,\mu):=&\left\{(x,y)\in V: 0<x<r, \frac{1}{\lambda}e(x)<y<\tilde{e}^{-1}(\mu x)\right\}\\
=&\left\{(x,y)\in [0,r]\times [0,r_2]: e^{-1/|x|^p}< \lambda^q y^q, e^{-1/|y|^{\tilde{p}}}<\mu^{\tilde{q}} x^{\tilde{q}}\right\}.
\end{split}
\end{equation}
Notice that $U(\lambda,\mu)$ is not empty.
Let us consider the integral
\begin{equation}\label{eqn:8.5}
W(\sigma)
=\int_{U(\lambda,\mu)}
x^{a\sigma}y^{b\sigma}dxdy,
\end{equation}
where $a, b$ are as in \eqref{eqn:8.1}.
\begin{lemma}\label{lem:8.4}
The following hold.
\begin{enumerate}
\item
If $p>1-a/b$, then
\begin{equation}
\lim_{\sigma\to -1/b+0}(b\sigma+1)^{1-\frac{1-a/b}{p}}\cdot W(\sigma)
=\hat{C},
\end{equation}
where $\hat{C}$ is as in \eqref{eqn:3.10}.
\item
If $p=1-a/b$, then
\begin{equation}
\lim_{\sigma\to -1/b+0}|\log (b\sigma+1)|^{-1}\cdot W(\sigma)=\frac{1}{pq}.
\end{equation}
\end{enumerate}
\end{lemma}

\begin{proof}
For simplicity, we denote $X:=b\sigma+1$.
We decompose 
$W(\sigma)$ as follows:
\begin{equation}\label{eqn:8.12}
\begin{split}
W(\sigma)
&=\int_{0}^{r}x^{a\sigma}\left(\int_{e(x)/\lambda}^{\tilde{e}^{-1}(\mu x)}
y^{b\sigma}dy\right)dx\\
&=W_1(\sigma)+W_2(\sigma)
\end{split}
\end{equation}
with
\begin{equation}
\begin{split}
W_1(\sigma)
&=\int_{0}^{r}x^{a\sigma}\left(\int_{e(x)/\lambda}^{1/\lambda}
y^{b\sigma}dy\right)dx,\\
W_2(\sigma)
&=\int_{0}^{r}x^{a\sigma}\left(\int_{1/\lambda}^{\tilde{e}^{-1}(\mu x)}
y^{b\sigma}dy\right)dx.
\end{split}
\end{equation}

A simple computation gives
\begin{equation}\label{eqn:8.14}
\begin{split}
W_1(\sigma)&=\frac{1}{X\lambda^{X}}\int_{0}^{r}
x^{-a/b+aX/b}\left(1-\left(e(x)\right)^{X}\right)dx.
\end{split}
\end{equation}
Applying Lemma \ref{lem:4.4} (i), (ii) to \eqref{eqn:8.14},
we have that
if $p>1-a/b$, then
\begin{equation}\label{eqn:8.16}
\lim_{X \to +0}X^{1-\frac{1-a/b}{p}}\cdot W_1(\sigma)
=\frac{1}{q^{{(1-a/b)/p}}(1-a/b)}\Gamma\left(1-\frac{1-a/b}{p}\right),
\end{equation}
and if $p=1-a/b$, then
\begin{equation}\label{eqn:8.17}
\lim_{X \to +0}|\log X|^{-1}\cdot W_1(\sigma)=\frac{1}{pq}.
\end{equation}
On the other hand,
Lebesgue's convergence theorem gives that
\begin{equation}\label{eqn:8.15}
\begin{split}
\lim_{X \to +0}W_2(\sigma)
&=\int_{0}^{r}x^{-a/b}\left(\int_{1/\lambda}^{\tilde{e}^{-1}(\mu x)}
y^{-1}dy\right)dx\\
&=\int_{0}^{r}x^{-a/b}\left(\log\left(\tilde{e}^{-1}(\mu x)\right)+\log \lambda \right)dx\\
&=\frac{\tilde{p}}{\tilde{q}\mu^{-a/b+1}}
\int_{0}^{\tilde{e}^{-1}(\mu r)}
u^{-\tilde{p}-1}(\log u+ \log \lambda)(\tilde{e}(u))^{-a/b+1}du.
\end{split}
\end{equation}
We remark that the last equality is obtained from changing the integral variable by $u=\tilde{e}^{-1}(\mu x)$, and the last integral in \eqref{eqn:8.15} converges
since $-a/b+1>0$ and $\tilde{e}(\cdot)$ is flat at the origin.

Combining \eqref{eqn:8.12}, \eqref{eqn:8.16}, \eqref{eqn:8.17} and \eqref{eqn:8.15}
gives that
if $p>1-a/b$, then
\begin{equation}
\begin{split}
\lim_{\sigma \to -1/b+0}(b\sigma+1)^{1-\frac{1-a/b}{p}}\cdot W(\sigma)
&=
\lim_{\sigma \to -1/b+0}(b\sigma+1)^{1-\frac{1-a/b}{p}}\cdot W_1(\sigma)\\
&=\frac{1}{q^{{(1-a/b)/p}}(1-a/b)}\Gamma\left(1-\frac{1-a/b}{p}\right),
\end{split}
\end{equation}
and if $p=1-a/b$, then
\begin{equation}
\begin{split}
\lim_{\sigma \to -1/b+0}|\log (b\sigma+1)|^{-1}\cdot W(\sigma)
&=\lim_{\sigma \to -1/b+0}|\log (b\sigma+1)|^{-1}\cdot W_1(\sigma)\\
&=\frac{1}{pq}.
\end{split}
\end{equation}
\end{proof}

\begin{proof}[Proof of Theorem \ref{thm:8.1}]
Let $\lambda, \mu$ be small positive real numbers
and $U(\lambda,\mu)$ as in \eqref{eqn:8.4.0}.
Then we have for $\sigma<0$
\begin{equation}\label{eqn:8.8}
\begin{split}
Z(\sigma)
&>\int_{U(\lambda,\mu)}
\left|x^ay^b+x^a y^{b-q}e^{-1/|x|^p}+x^{a-\tilde{q}}y^{b}e^{-1/|y|^{\tilde{p}}}\right|^{\sigma}dxdy\\
&>\left(1+\lambda^q+\mu^{\tilde{q}}\right)^{\sigma}\cdot W(\sigma),
\end{split}
\end{equation}
where $W(\sigma)$ is as in \eqref{eqn:8.5}.
On the other hand,
we have the inequality
\begin{equation}\label{eqn:8.4}
Z(\sigma)<\tilde{W}(\sigma)
\end{equation}
with 
\begin{equation}
\begin{split}
\tilde{W}(\sigma)=\int_{V}
\left|x^ay^b+x^ay^{b-q}e^{-1/|x|^p}\right|^{\sigma}
dxdy
\end{split}
\end{equation}
since $\sigma<0$ and $q, \tilde{q}$ are even.
The asymptotic limit of $\tilde{W}(\sigma)$ have been already investigated.
\begin{lemma}[Theorem 3.1 of \cite{kn19}]\label{lem:8.2}
The following hold.
\begin{enumerate}
\item
If $p>1-a/b$, 
then
\begin{equation}
\lim_{\sigma\to-1/b+0}
(b\sigma+1)^{1-\frac{1-a/b}{p}}\cdot \tilde{W}(\sigma)
=\hat{C},
\end{equation}
where $\hat{C}$ is as in \eqref{eqn:3.10}.
\item
If $p=1-a/b$, 
then 
\begin{equation}
\lim_{\sigma\to-1/b+0}
|\log (b\sigma+1)|^{-1}\cdot \tilde{W}(\sigma)
=\frac{1}{pq}.
\end{equation}
\item
If $0<p< 1-a/b$, then the limit of 
$\tilde{W}(\sigma)$ as 
$\sigma\to-1/b+0$ exists. 
\end{enumerate}
\end{lemma}

\begin{remark}
(i) In \cite{kn19}, the constant $\hat{C}$	is represented as
$\hat{C}=\int_0^{\infty}x^{-a/b}\left(1-e^{-1/(qx^p)}\right)dx$.

(ii)
The limit of $\tilde{W}(\sigma)$ as $\sigma\to-1/b+0$ in the case where $0<p<1-a/b$
is precisely investigated in \cite{kn19}.
\end{remark}

We assume $p>1-a/b$.
From \eqref{eqn:8.8}, \eqref{eqn:8.4},  
Lemmas \ref{lem:8.4} and \ref{lem:8.2},
we have the inequalities
\begin{equation}
\begin{split}
(1+\lambda^q+\mu^{\tilde{q}})^{-1/b}\cdot 
\hat{C}
&\leq \liminf_{\sigma\to -1/b+0}(b\sigma+1)^{1-\frac{1-a/b}{p}}\cdot Z(\sigma)\\
&\leq \limsup_{\sigma\to -1/b+0}(b\sigma+1)^{1-\frac{1-a/b}{p}}\cdot Z(\sigma)\\
&\leq \hat{C},
\end{split}
\end{equation}
where $\hat{C}$ is as in \eqref{eqn:3.10}.
Noticing that $Z(\sigma)$ is independent of $\lambda, \mu$ and
considering the limit as $\lambda\to 0$ and $\mu \to 0$,
we have \eqref{eqn:8.2}.
The case (ii) can be similarly shown.

Let us consider the case where $0<p<1-a/b$.
From \eqref{eqn:8.4} and Lemma \ref{lem:8.2},
we have the boundedness of $\limsup_{\sigma\to 1-/b+0}Z(\sigma)$.
Applying Lemma \ref{lem:4.5} to $Z(\sigma)$,
we obtain the assertion. Thus the proof of the theorem is complete.
\end{proof}

\subsection{Proof of Theorem \ref{thm:3.3}}

Let $f$ be as in \eqref{eqn:3.7} and $\varphi$ in \eqref{eqn:1.1}.
Under the assumption that $q,\tilde{q}$ are even, 
we see that $|f(\theta_1 x,\theta_2 y)|=|f(x,y)|$ for any $(x,y)\in\R^2$
and $\theta_1,\theta_2\in\{1,-1\}$. 
Using the orthant decomposition, we have
\begin{equation}\label{eqn:8.11}
Z_{f}(\varphi)(\sigma)=\sum_{\theta\in\{1,-1\}^2}\tilde{Z}_{f}(\varphi_{\theta})(\sigma)
\end{equation}
with
\begin{equation}
\tilde{Z}_{f}(\varphi)(\sigma)=\int_{\R_+^2}|f(x,y)|^{\sigma}\varphi_{\theta}(x,y)dxdy,
\end{equation}
where $\varphi_{\theta}$ is as in \eqref{eqn:7.7} with $\theta=(\theta_1,\theta_2)$.
We consider the integral $\tilde{Z}_{f}(\varphi)(\sigma)$.

First, we assume $p>1-a/b$.
For any $\epsilon>0$, 
there exist $r_1, r_2\in(0,1)$ such that
$\varphi(0,0)-\epsilon \leq 
\varphi(x,y)\leq 
\varphi(0,0)+\epsilon$
for $(x,y)\in V$,
where $V=\{(x,y)\in \R^2:0\leq x\leq r_1,0\leq y \leq r_2\}$.
These inequalities imply that
\begin{equation}\label{eqn:9.3}
\begin{split}
(\varphi(0,0)-\epsilon) \cdot 
Z(\sigma)
&\leq 
\tilde{Z}_f(\varphi)(\sigma)-
\int_{\R_+^2\setminus V}|f(x,y)|^{\sigma}\varphi(x,y)dxdy\\
&\leq
(\varphi(0,0)+\epsilon) \cdot 
Z(\sigma),
\end{split}
\end{equation}
where $Z(\sigma)$ is as in \eqref{eqn:8.1}.
We have
\begin{equation}\label{eqn:9.4}
\begin{split}
\int_{\R_+^2\setminus V}|f(x,y)|^{\sigma}\varphi(x,y)dxdy
=\int_{\R_+^2\setminus V}|x|^{(a-\tilde{q})\sigma} |y|^{(b-q)\sigma}|\tilde{f}(x,y)|^{\sigma}\varphi(x,y)dxdy,
\end{split}
\end{equation}
where $\tilde{f}(x,y)=x^{\tilde{q}}y^{q}+x^{\tilde{q}}e^{-1/|x|^p}+y^qe^{-1/|y|^{\tilde{p}}}$.
We remark that $\tilde{f}$ does not vanish on $\R_+^2\setminus V$,
which implies that
the integral converges when $\sigma>\max\{-1/(a-\tilde{q}),-1/(b-q)\}$.
In particular, we have
for $\sigma\geq -1/b$ that the integral in \eqref{eqn:9.4} converges
and the modulus of this integral is bounded by a positive constant
which is independent of $\sigma$.
As a result, 
the inequalities \eqref{eqn:9.3}, 
Theorem \ref{thm:8.1} and the orthant decomposition \eqref{eqn:8.11} easily imply (i) in Theorem \ref{thm:3.3}.
The case (ii) is analogously proven.

In the case (iii),
we have from Theorem \ref{thm:8.1} that the limit of $Z(\sigma)$ as $\sigma \to -1/b+0$ exists.
Applying Lemma \ref{lem:4.6} to $\tilde{Z}_{f}(\varphi)(\sigma)$
and the orthant decomposition \eqref{eqn:8.11} give the assertion,
which completes the proof of the theorem.

\bigskip
{\it Acknowledgments.}\quad 
The author would like to express his sincere gratitude to Joe Kamimoto and Atsushi Yamamori for many valuable discussions and comments, 
and to the referee for his/her careful reading of the manuscript and giving the author many valuable comments.

\end{document}